\documentclass[preprint,11pt]{elsarticle}

\usepackage{amssymb}
\usepackage{amsthm}

\usepackage{amsfonts, amsmath, mathrsfs}
\usepackage{subfig}
\usepackage{xcolor}

\newtheorem{theo}{Theorem}[section]

\newtheorem{prop}[theo]{Proposition}
\newtheorem{lemma}[theo]{Lemma}
\newtheorem{example}[theo]{Example}
\newtheorem{assumption}[theo]{Assumption}
\newtheorem{defn}[theo]{Definition}
\newtheorem{remark}[theo]{Remark}

\newcommand\norm[1]{\left\lVert#1\right\rVert}

\allowdisplaybreaks

\journal{Stochastic Processes and their Applications}

\begin{document}

\begin{frontmatter}



\title{PageRank on inhomogeneous random digraphs}


\author[label1]{Jiung Lee}
\author[label2]{Mariana Olvera-Cravioto}
\address[label1]{Department of Industrial Engineering and Operations Research, University of California, Berkeley}
\address[label2]{Department of Statistics and Operations Research, University of North Carolina at Chapel Hill}



\begin{abstract}
We study the typical behavior of Google's PageRank algorithm on inhomogeneous random digraphs, including directed versions of the Erd\H os-R\'enyi model, the Chung-Lu model, the Poissonian random graph and the generalized random graph. Specifically, we show that the rank  of a randomly chosen vertex converges weakly to the attracting endogenous solution to the stochastic fixed-point equation 
$$\mathcal{R}\stackrel{\mathcal{D}}{=}\sum_{i=1}^{\mathcal{N}}\mathcal{C}_i \mathcal{R}_i+\mathcal{Q},$$ 
where  $(\mathcal{N},\mathcal{Q},\{\mathcal{C}_i\}_{i \geq 1})$ is a real-valued vector with $\mathcal{N}\in\mathbb{N}$, and the $\{\mathcal{R}_i\}$ are i.i.d.~copies of $\mathcal{R}$, independent of $(\mathcal{N},\mathcal{Q},\{\mathcal{C}_i\}_{i \geq 1})$; $\stackrel{\mathcal{D}}{=}$ denotes equality in distribution. This result provides further evidence of the power-law behavior of PageRank on graphs whose in-degree distribution follows a power law.
\end{abstract}


\end{frontmatter}


\section{Introduction}

In the recent decades, a growing amount of data and computer power has motivated the development of algorithms capable of efficiently organizing and analyzing large data sets. In many cases, this data is highly interconnected, and can be represented in the form of complex networks. Some important examples include the Internet and the World Wide Web, telecommunication networks, electrical power grids, protein-protein interactions, and the various social networks that have become an integral part of our society. Interestingly, many of these networks share some basic characteristics that we have learned to expect, such as short typical distances between nodes, known as the {\em small-world property}, and highly variable degrees whose distributions follow a power-law, known as the {\em scale-free} property. Of special interest is the problem of identifying relevant or central nodes in these networks.

We focus on the analysis of a general form of Google's PageRank algorithm \cite{Brin_Page}, which was originally created to rank webpages in the World Wide Web. PageRank is a popular algorithm for ranking nodes in complex networks due to its ability to efficiently identify important/relevant nodes. Its typical behavior on scale-free directed complex networks has also been an important research topic, since abundant empirical evidence suggests that the distribution of the ranks produced by PageRank follows a power-law distribution with the same tail index as the in-degree distribution  \cite{Pan_Rag_Upfal, Avra_Leb, Lit_Sch_Vol, Volk_Litv_10, Vol_Lit_Don, Jel_Olv_10}. The first rigorous proof of why this power-law behavior is observed was given in \cite{chenetal}, where it was shown that the rank of a randomly chosen node in a graph generated via the directed configuration model \cite{Chen_Olv_13} has a limiting distribution exhibiting power-law tails whenever the in-degree distribution is scale-free. Here, we extend this analysis to a different class of random graph models that includes as special cases directed versions of classical models such as the Erd\H os-R\'enyi graph \cite{Erdos, Gilbert, Austin, Janson, Bollobas2, Durrett1}, the Chung-Lu model \cite{Chunglu, Chunglu2, Chunglu3, Chunglu4, Lu}, the Poissonian random graph \cite{Norros, Hofstad1, Esker_Hofs_Hoog} and the generalized random graph \cite{Hofstad1, Brittonetal, Esker_Hofs_Hoog}.  Since the scale-free property of the degrees in these models (i.e., their power-law behavior)  is due to node-specific attributes which can also be used to influence the rankings produced by PageRank, we believe they provide a more natural way of modeling and understanding the behavior of ranking algorithms on complex networks than the directed configuration model.

As is the case for the directed configuration model, the power-law behavior of the ranks produced by PageRank can be explained by arguing that the limiting distribution of the rank of a randomly chosen node can be written in terms of the attracting endogenous solution to a stochastic-fixed point equation (SFPE) of the form
$$
\mathcal{R} \stackrel{\mathcal{D}}{=} \sum_{i=1}^{\mathcal{N}} \mathcal{C}_i \mathcal{R}_i + \mathcal{Q},
$$
where the $\{\mathcal{R}_i\}$ are i.i.d.~copies of $\mathcal{R}$, independent of the vector $(\mathcal{N}, \mathcal{Q}, \{ \mathcal{C}_i \}_{i \geq 1})$, with the $\{\mathcal{C}_i\}$ i.i.d.~and independent of $(\mathcal{N}, \mathcal{Q})$. The random variable $\mathcal{N}$ corresponds to the in-degree distribution of the network being analyzed, and the power law behavior of the rank distribution follows from the known asymptotic equivalence 
$$P(\mathcal{R} > x) \sim K P(\mathcal{N} > x), \qquad x \to \infty,$$
for some constant $0 < K < \infty$, when $\mathcal{N}$ has a power-law distribution \cite{Volk_Litv_10, Jel_Olv_10}.  Theorem~\ref{T.MainPageRank} in this paper shows that the same type of representation holds for the family of inhomogeneous random digraphs studied here, provided the in-degree and out-degree of the same vertex are asymptotically independent. Therefore, the results in this paper provide further evidence of the power-law behavior of PageRank in scale-free directed networks.

The remainder of the paper is organized as follows. Section~\ref{S.Model} describes the family of inhomogeneous random digraphs mentioned above, and includes some of its most basic properties, in particular, its ability to generate inhomogeneous directed graphs with a wide range of degree distributions, including scale-free ones. In Section~\ref{S.PageRank} we state our main result on the distribution of the ranks produced by a generalized form of PageRank, including the description of a three-step approach towards its proof, and in Section~\ref{S.Proofs} we provide all the proofs. Numerical experiments as well as additional results for specific models are included in the Appendix.

\section{A family of inhomogeneous random digraphs} \label{S.Model}

As mentioned in the introduction, the fact that many of the complex networks in the real world exhibit highly variable degrees, often with tails that appear to follow a power law, motivates our interest in random graph models capable of generating inhomogeneous degrees. One model that produces graphs from any prescribed (graphical) degree sequence is the configuration or pairing model \cite{bollobas, Hofstad1}, which assigns to each vertex in the graph a number of half-edges equal to its target degree and then randomly pairs half-edges to connect vertices. The resulting graph, when the pairing process does not create self-loops or multiple edges, is known to have the distribution of a uniformly chosen graph among all graphs having the prescribed degree sequence.  If one chooses this degree sequence according to a power-law, one immediately obtains a scale-free graph. 

Alternatively, one could think of obtaining the scale-free property (power-law degree distribution) as a consequence of how likely different nodes are to have an edge between them. In the spirit of the classical Erd\H os-R\'enyi graph \cite{Erdos, Gilbert, Austin, Janson, Bollobas2, Durrett1}, we assume that whether there is an edge between vertices $i$ and $j$ is determined by a coin-flip, independently of all other edges. Unfortunately, this elegant and simple rule is known to produce highly homogeneous degrees, Poisson distributed in the limit, making it inappropriate for modeling most real-world networks. Several models capable of producing graphs with inhomogeneous degrees while preserving the independence among edges have been suggested in the recent literature, including: the Chung-Lu model \cite{Chunglu, Chunglu2, Chunglu3, Chunglu4, Lu}, the Norros-Reittu model (or Poissonian random graph) \cite{Norros, Hofstad1, Esker_Hofs_Hoog}, and the generalized random graph \cite{Hofstad1, Brittonetal, Esker_Hofs_Hoog}, to name a few. In all of these models, the inhomogeneity of the degrees is created by allowing the success probability of each coin-flip to depend on the ``attributes" of the two vertices being connected; the scale-free property can then be obtained by choosing the attributes according to a power-law. We briefly mention that it was shown in \cite{Esker_Hofs_Hoog} that all these models also exhibit the small-world property, i.e., small typical distances between vertices, hence, we expect the same to be true of their directed counterparts.

We now give a precise description of the family of directed random graphs that we study in this paper, which includes as special cases the directed versions of all the models mentioned above. Throughout the paper we refer to a directed graph $\mathcal{G}(V_n, E_n)$ on the vertex set $V_n = \{1 , 2, \dots, n\}$ simply as a {\em random digraph} if the event that edge $(i,j)$ belongs to the set of edges $E_n$ is independent of all other edges. 

In order to obtain inhomogeneous degree distributions, to each vertex $i \in V_n$ we assign a {\em type} ${\bf W}_i = (W_i^-, W_i^+) \in \mathbb{R}_+^2$, which will be used to determine how likely vertex $i$ is to have inbound/outbound neighbors\footnote{The $-$ and $+$ superscripts refer to the inbound or outbound nature of edges in the graph, and are not related to the positive and negative parts of a real number.}. The sequence of types $\{{\bf W}_i: i \geq 1\}$ is assumed to have a limiting behavior, in the sense that the empirical joint distribution satisfies:
\begin{equation} \label{eq:TypeDistr}
F_n(u,v) = \frac{1}{n} \sum_{i=1}^n 1(W_i^- \leq u, \, W_i^+ \leq v) \xrightarrow{P} F(u,v), \qquad \text{as } n \to \infty,
\end{equation}
for all continuity points of some distribution $F$, where $F$ is defined on the space $\mathcal{S} = \mathbb{R}_+^2$ and $\stackrel{P}{\to}$ denotes convergence in probability. Let $\mathscr{F}_n =\sigma( {\bf W}_i: 1 \leq i \leq n)$, and define $\mathbb{P}_n(\cdot) = P( \cdot | \mathscr{F}_n)$ and $\mathbb{E}_n [ \cdot] = E[ \cdot | \mathscr{F}_n]$ to be the conditional probability and conditional expectation, respectively, given the type sequence.  Later in Section~\ref{S.PageRank} we will enlarge the type vectors to include additional vertex attributes.

\begin{remark}
 In general, depending on the nature of the type sequence (e.g., a deterministic sequence of numbers), it may be necessary to consider a double sequence $\{{\bf W}_i^{(n)}: i \geq 1, \, n \geq 1\}$ in order to satisfy \eqref{eq:TypeDistr}. In practice, an easy way to avoid the need for considering double sequences is to assume that the  type sequence $\{{\bf W}_i : i \geq 1\}$ consists of i.i.d.~observations from distribution $F$. 
\end{remark}

We now define our family of random digraphs using the conditional probability, given the type sequence, that edge $(i,j) \in E_n$, 
\begin{equation} \label{eq:EdgeProbabilities}
p_{ij}^{(n)} \triangleq \mathbb{P}_{n} \left( (i,j) \in E_n \right) = 1 \wedge  \frac{W_i^+ W_j^-}{\theta n} (1 + \varphi_{n}({\bf W}_i, {\bf W}_j)) , \qquad 1 \leq i \neq j \leq n,
\end{equation}
where $-1 < \varphi_{n}({\bf W}_i, {\bf W}_j) $ a.s., and $n^{-1} \sum_{i=1}^n (W_i^- + W_i^+) \stackrel{P}{\to} \theta > 0$ as $n \to \infty$. 

We point out that the term $\varphi_{n}({\bf W}_i, {\bf W}_j)= \varphi(n, {\bf W}_i, {\bf W}_j, \mathscr{W}_n)$ may depend on the entire sequence $\mathscr{W}_n \triangleq \{{\bf W}_i: 1 \leq i \leq n\}$, on the types of the vertices $(i,j)$, or exclusively on $n$. Here and in the sequel, $x \wedge y = \min\{x,y\}$ and $x \vee y = \max\{x, y\}$. In the context of \cite{bollobasetal2}, definition \eqref{eq:EdgeProbabilities} corresponds to the so-called rank-1 kernel, i.e., $\kappa({\bf W}_i, {\bf W}_j) = \kappa_+({\bf W}_i) \kappa_-({\bf W}_j)$, with $\kappa_+({\bf W}) = W^+/\sqrt{\theta}$ and $\kappa_-({\bf W}) = W^-/\sqrt{\theta}$. 

\begin{example} \label{E.IRG}
Directed versions of some known random graph models covered by \eqref{eq:EdgeProbabilities}:
\begin{itemize}
\item Directed Erd\H os-R\'enyi model: 
$$p_{ij}^{(n)} = \frac{\lambda}{2n}, \qquad 1 \leq i \neq j \leq n,$$
for $\lambda > 0$, which corresponds to taking $W_i^- = W_i^+ = \lambda$ for all $1 \leq i \leq n$. This graph produces homogeneous graphs with Poisson$(\lambda)$ degrees. Here, $\varphi_{n}({\bf W}_i, {\bf W}_j) \equiv 0$ for all $1 \leq i \neq j \leq n$. 

\item Directed Chung-Lu model:
$$p_{ij}^{(n)} = \frac{W_i^+ W_j^-}{L_n} \wedge 1, \qquad 1 \leq i \neq j \leq n,$$
where $L_n = \sum_{i=1}^n (W_i^- + W_i^+)$. This model is defined for any nonnegative sequences $\{W_i^-: i \geq 1\}$ and $\{W_i^+: i \geq 1\}$ possessing some limiting distributions, e.g., power-laws, usually with finite second moments. Here, $\varphi_{n}({\bf W}_i, {\bf W}_j) = \theta n/L_n - 1$ for all $1 \leq i \neq j \leq n$. 

\item Directed generalized random graph:
$$p_{ij}^{(n)} = \frac{W_i^+ W_j^-}{L_n + W_i^+ W_j^-} , \qquad 1 \leq i \neq j \leq n,$$
where $L_n$ is defined as above. Since the ratios in the definition of $p_{ij}^{(n)}$ are self-normalized, it provides a more natural model for graphs with infinite variance degrees. Here, $\varphi_{n}({\bf W}_i, {\bf W}_j)  = \theta n/(L_n + W_i^+ W_j^-)  - 1$ for $1 \leq i \neq j \leq n$. 

\item Directed Poissonian random graph or Norros-Reittu model: 
$$p_{ij}^{(n)} = 1 - e^{-W_i^+ W_j^-/L_n}, \qquad 1 \leq i \neq j \leq n,$$
where $L_n$ is defined as above. Here, $\varphi_{n}({\bf W}_i, {\bf W}_j)  = (1 - e^{-W_i^+ W_j^-/L_n}) \theta n/(W_i^+ W_j^-) - 1$ for $1 \leq i\neq j \leq n$. 

\end{itemize}
\end{example}

From a modeling perspective, one can think of $W_i^+$ as an attribute of vertex $i$ that determines how likely it is for it to have outbound neighbors, and $W_i^-$ as an attribute that indicates its popularity, or likelihood that other vertices may have edges pointing towards it. In other words, $W_i^+$ controls the out-degree of vertex $i$ and $W_i^-$ its in-degree. In applications, e.g., the World Wide Web, these two attributes can be used to model how trustworthy a webpage is, how valuable/relevant is its content, or how carefully it chooses the webpages it references.

\subsection{Degree distributions}

Our first result in the paper establishes that the family of random digraphs defined via \eqref{eq:EdgeProbabilities} produces inhomogeneous graphs whose degree distribution can be modeled through that of the type distribution.  
A verification of all our assumptions for each of the models in Example~\ref{E.IRG} when the type sequence $\{ {\bf W}_i: i \geq 1\}$ consists of i.i.d.~random vectors is included in \ref{Appx.IIDsequences}.

\begin{assumption} \label{A.Types}
Let $\mathcal{G}(V_n, E_n)$ be a random digraph having type sequence $\{ {\bf W}_i: i\geq 1\}$ and edge probabilities given by \eqref{eq:EdgeProbabilities}. Suppose further that:
\begin{enumerate}
\item[a)] The type sequence $\{ {\bf W}_i: i\geq 1\}$ satisfies \eqref{eq:TypeDistr}.
\item[b)] The following limits hold in probability:
\begin{align*}
E[W^-]  &= \lim_{n \to \infty} \frac{1}{n} \sum_{i=1}^n W_i^-  \qquad \text{and} \qquad E[W^+]  = \lim_{n \to \infty} \frac{1}{n} \sum_{i=1}^n W_i^+,
\end{align*}
with $\theta = E[W^- + W^+] < \infty$. 
\item[c)] $\displaystyle \mathcal{E}_n \triangleq \frac{1}{n} \sum_{i=1}^n \sum_{1 \leq j \leq n, j \neq i} |p_{ij}^{(n)} - (r_{ij}^{(n)} \wedge 1)| \xrightarrow{P} 0$ as $n \to \infty$, where $r_{ij}^{(n)} = W_i^+ W_j^-/(\theta n)$. 
\item[d)] $\displaystyle \frac{1}{n} \sum_{i=1}^n \sum_{1 \leq j \leq n, j \neq i} \sum_{1 \leq k \leq n, k\neq i} p_{ji}^{(n)} p_{ik}^{(n)} \xrightarrow{P} E[W^- W^+] E[W^+] E[W^-]/\theta^2$ as $n \to \infty$, with \newline $E[W^- W^+] < \infty$. 
\end{enumerate}
\end{assumption}

We point out that Assumption~\ref{A.Types} implies that $|\varphi_n({\bf W}_i, {\bf W}_j)| \xrightarrow{P} 0$ as $n \to \infty$ for any $i,j \geq 1$. 

We now define the in-degree and out-degree of vertex $i \in V_n$ according to 
$$D_i^- = \sum_{j \in V_n, \, j \neq i} X_{ji} \qquad \text{and} \qquad  D_i^+ = \sum_{j \in V_n, \ j \neq i} X_{ij}, $$
respectively, where $X_{ij} = 1( (i,j) \in E_n)$ is the indicator function of whether edge $(i,j)$ is present in the graph. Note that from the independent edges assumption, we have that the $\{X_{ij}: 1 \leq i \neq j \leq n\}$ form a sequence of independent Bernoulli random variables, with $\mathbb{P}_{n}( X_{ij} = 1) = p_{ij}^{(n)}$. 

The following theorem provides the distribution of the in-degree and out-degree of a randomly chosen vertex, i.e., a typical vertex, in a graph generated via our model. Its proof is given in Section~\ref{SS.DegreeProof}.

\begin{theo} \label{T.MixedPoisson}
Under Assumption~\ref{A.Types} (a)-(c), the degrees  $(D_{\xi}^-, D_{\xi}^+)$ of a randomly chosen vertex in $\mathcal{G}(V_n, E_n)$ satisfy, as $n \to \infty$, 
$$\sup_{A \subseteq \mathbb{N}^2} \left| \mathbb{P}_n\left( (D_{\xi}^-, D_{\xi}^+) \in A \right) - P((Z^-, Z^+) \in A) \right| \xrightarrow{P} 0,$$
and 
$$\mathbb{E}_n [ D_\xi^\pm] \stackrel{P}{\longrightarrow} E[ Z^\pm] = \frac{E[W^-] E[W^+]}{\theta},$$
where $Z^-$ and $Z^+$ are mixed Poisson random variables with mixing parameters, $\frac{E[W^+]}{\theta} W^- $ and $ \frac{E[W^-]}{\theta} W^+$, respectively, with $Z^-$ and $Z^+$ conditionally independent given $(W^-, W^+) $. Moreover, if Assumption~\ref{A.Types} (a)-(d) holds then, as $n \to \infty$,
$$ \mathbb{E}_n[ D_\xi^- D_\xi^+] \stackrel{P}{\longrightarrow} E [ Z^- Z^+] = \frac{E[W^- W^+] E[W^-] E[W^+]}{\theta^2}.$$
\end{theo}

\begin{remark} \label{R.PowerLaw}
To relate this result with scale-free graphs where at least one of the degree distributions, usually the in-degree, follows a power-law, we point out that when $W^-$ ( $W^+$) has a regularly varying distribution with index $-\alpha < -1$, i.e., $P(W^- > x) = x^{-\alpha} L(x)$ for some slowly varying function $L$, then, by Proposition~8.4 in \cite{Grandell}, we have that $Z^-$ ($Z^+$) is also regularly varying with the same index. Furthermore, it can be shown that if $(W^-, W^+)$ is jointly regularly varying (possibly in the non-standard sense defined in \cite{samorodnitskyetal}), then so is $(Z^-, Z^+)$, however, for our analysis of PageRank we will impose that $W^-$ and $W^+$ be independent, so only the marginal distributions of $Z^-$ and $Z^+$ are relevant to our main result. 
\end{remark}

 \section{Generalized PageRank} \label{S.PageRank}
 
We now move on to the analysis of the typical behavior of the PageRank algorithm on the family of inhomogeneous random digraphs described in Section~\ref{S.Model}. Our main result shows that the distribution of the ranks produced by the algorithm converges to that of the attracting endogenous solution, $\mathcal{R}$, to a linear SFPE. Moreover, since the behavior of $\mathcal{R}$ is known to follow a power-law when the limiting in-degree distribution does, our theorem provides further evidence of the universality of the so-called ``power-law hypothesis" on scale-free complex networks \cite{chenetal}. For completeness, we give below a brief description of the algorithm, which is well-defined for any directed graph $\mathcal{G}(V_n, E_n)$ on the vertex set $V_n = \{1, 2, \dots, n\}$ with edges in the set $E_n$. 

Let $D_i^-$ and $D_i^+$ denote the in-degree and out-degree, respectively, of vertex $i$ in $\mathcal{G}(V_n, E_n)$.  We refer to the sequence $\{ (D_i^-, D_i^+): 1 \leq i \leq n\}$ as the bi-degree sequence of the graph $\mathcal{G}(V_n, E_n)$. The generalized PageRank vector ${\bf r} = (r_1, \dots, r_n)$ is the unique solution to the following system of equations:
\begin{equation} \label{eq:genPageRank}
r_i = \sum\limits_{(j,i)\in E_n} \frac{\zeta_j}{D_j^+} \cdot r_j + q_i,\quad i=1, \dots,n,
\end{equation}
where ${\bf q} = (q_1, \dots, q_n)$ is known as the personalization or teleportation vector (usually, a probability vector), and the $\{\zeta_i\}$ are referred to as the damping factors. In the original formulation of PageRank \cite{Brin_Page}, the personalization values and the damping factors are given, respectively, by $q_i = (1- c)/n$ and $\zeta_i = c$ for all $1 \leq i \leq n$; the constant $c \in (0,1)$ is known as the ``damping factor". The formulation given in \cite{chenetal} is more general, and it allows any choice for both the personalization values and the damping factors, provided that $\max_{1 \leq i \leq n} |\zeta_i| \leq c <  1$. We refer the reader to \S1.1 in \cite{chenetal} for further details on the history of PageRank, its applications, and a matrix representation of the solution ${\bf r}$ to \eqref{eq:genPageRank}. 

In order to analyze {\bf r}  on directed complex networks, we first eliminate the dependence on the size of the graph by computing the scale free ranks $(R_1, \dots, R_n) = {\bf R} \triangleq n {\bf r}$, which corresponds to solving:
\begin{equation} \label{eq:scaleFreePageRank}
R_i = \sum\limits_{(j,i)\in E_n} C_j R_j + Q_i, \quad i = 1, \dots, n,
\end{equation}
where $Q_i = q_i n$ and $C_j = \zeta_j/D_j^+$. We refer to the $\{C_j\}$ as the weights. 

On scale-free graphs, i.e., where the in-degree sequence (or both the in-degree and out-degree sequences) follow a power law distribution, the power law hypothesis states that the distribution of the ranks $\{R_i: 1 \leq i \leq n\}$ will also have a power-law with the same index as that of the in-degrees. The first approach towards a proof of this phenomenon was given in \cite{Vol_Lit_Don, Lit_Sch_Vol, Jel_Olv_10}, where the tree heuristic commonly used in the analysis of locally tree-like random graphs yields a stochastic fixed-point equation of the form
\begin{equation} \label{eq:SFPE}
\mathcal{R} \stackrel{\mathcal{D}}{=} \sum_{j=1}^{\mathcal{N}} \mathcal{C}_j \mathcal{R}_j + \mathcal{Q},
\end{equation}
where $\mathcal{N}$ is a random variable distributed according to the limiting in-degree distribution of the graph, $\mathcal{Q}$ has the limiting distribution of the personalization values, the weights $\{ \mathcal{C}_j\}$ are i.i.d.~and independent of $(\mathcal{N}, \mathcal{Q})$, and are size-biased versions of the weights $\{C_j\}$ in \eqref{eq:scaleFreePageRank}, and the $\{\mathcal{R}_i\}$ are i.i.d.~copies of $\mathcal{R}$.  The connection between \eqref{eq:scaleFreePageRank} and \eqref{eq:SFPE} can be understood by interpreting $\mathcal{R}$ as the rank of a randomly chosen  vertex, with $(\mathcal{N}, \mathcal{Q})$ denoting its in-degree and personalization value, respectively, and then arguing that, provided the neighborhood of the chosen vertex looks locally like a tree, the ranks of its inbound neighbors should have the same distribution as $\mathcal{R}$. That the weights $\{\mathcal{C}_j\}$ in \eqref{eq:SFPE} are different from the $\{C_j\}$ appearing in \eqref{eq:scaleFreePageRank}, and are instead size-biased versions of them,  follows from the observation that vertices with high out-degrees are more likely to be the neighbors of the randomly chosen vertex. 

The heuristic described above was first made rigorous in \cite{chenetal}, where it was shown that on graphs generated via the directed configuration model \cite{Chen_Olv_13}, the rank of a randomly chosen vertex converges in distribution, as the size of the graph grows to infinity, to a random variable 
\begin{equation*}
\mathcal{R}^* = \sum_{i=1}^{\mathcal{N}_0} \mathcal{C}_i \mathcal{R}_i + \mathcal{Q}_0,
\end{equation*}
where the $\{ \mathcal{R}_i\}$ are i.i.d.~copies of the attracting endogenous solution to \eqref{eq:SFPE}, and are independent of $(\mathcal{N}_0, \mathcal{Q}_0, \{ \mathcal{C}_i\}_{i \geq 1})$. The vector $(\mathcal{N}_0, \mathcal{Q}_0)$ may have a different distribution from that of $(\mathcal{N}, \mathcal{Q})$ in \eqref{eq:SFPE} depending on how we choose the first vertex (it has the same distribution as $(\mathcal{N}, \mathcal{Q})$ when the first vertex is chosen uniformly at random and the in-degree and out-degree are asymptotically independent, which is consistent with the approach we take here). That the solution $\mathcal{R}$ to \eqref{eq:SFPE} has a power-law distribution when $\mathcal{N}$ does has been the topic of a number of papers \cite{Volk_Litv_10, Jel_Olv_10, Jel_Olv_12, Jel_Olv_12_2}, and together with the results in \cite{chenetal} (see Theorems~6.4 and 6.6) provides the first proof of the power-law hypothesis on a complex network. We now show that a similar result also holds for the family of inhomogeneous random digraphs considered here.

\subsection{PageRank on inhomogeneous random digraphs} \label{SS.MainPageRank}

As with the analysis done in \cite{chenetal} on the directed configuration model, the key idea is to couple the rank of a randomly chosen vertex with the rank of the root node of a tree, in this case, a multi-type branching process. In order to incorporate vertex information used by the algorithm, as described by \eqref{eq:genPageRank}, we expand the type of vertex $i$ to be of the form  ${\bf W}_i = (W^-_i, W^+_i, Q_i, \zeta_i) \in \mathbb{R}_+^2 \times \mathbb{R}^2 \triangleq \mathcal{S}$, where the sequence $\{ {\bf W}_i: i \geq 1\}$ satisfies
\begin{equation} \label{eq:ExtTypeDistr}
H_n(u,v,q,t) = \frac{1}{n} \sum_{i=1}^n 1(W_i^- \leq u, W_i^+ \leq v, Q_i \leq q, \zeta_i \leq t) \stackrel{P}{\longrightarrow} H(u,v,q,t), 
\end{equation}
as $n \to \infty$, for all continuity points of some distribution $H$. With some abuse of notation, we continue using $\mathscr{F}_n = \sigma( {\bf W}_i : 1\leq  i \leq n )$ to denote the sigma-algebra generated by $\{{\bf W}_i: 1 \leq i \leq n\}$, along with the corresponding conditional probability and expectation $\mathbb{P}_{n}(\cdot ) = P( \cdot | \mathscr{F}_n)$ and $\mathbb{E}_{n}[ \cdot] = E[ \cdot | \mathscr{F}_n]$. 

We now impose some assumptions on the extended type sequence $\{ {\bf W}_i: i \geq 1\}$.

\begin{assumption} \label{A.ExtendedTypes}
Let $\mathcal{G}(V_n, E_n)$ be a random digraph having type sequence $\{ {\bf W}_i: i\geq 1\}$ and edge probabilities given by \eqref{eq:EdgeProbabilities}. Suppose further that:
\begin{enumerate}
\item[a)] The extended type sequence $\{ {\bf W}_i: i\geq 1\}$ satisfies \eqref{eq:ExtTypeDistr}.
\item[b)] The following limits hold in probability:
\begin{align*}
E[W^-]  &= \lim_{n \to \infty} \frac{1}{n} \sum_{i=1}^n W_i^- , \qquad E[W^+]  = \lim_{n \to \infty} \frac{1}{n} \sum_{i=1}^n W_i^+, \\E[|Q|]  &= \lim_{n \to \infty} \frac{1}{n} \sum_{i=1}^n |Q_i| ,  \qquad \text{and} \qquad E[|\zeta|]  = \lim_{n \to \infty} \frac{1}{n} \sum_{i=1}^n |\zeta_i|,
\end{align*}
with $\theta = E[W^- + W^+] < \infty$ and $E[ |Q| + |\zeta|] < \infty$. 
\item[c)] $\displaystyle \mathcal{E}_n = \frac{1}{n} \sum_{i=1}^n \sum_{1 \leq j \leq n, j \neq i} |p_{ij}^{(n)} - (r_{ij}^{(n)} \wedge 1) | \xrightarrow{P} 0$ as $n \to \infty$, where $r_{ij}^{(n)} = W_i^+ W_j^-/(\theta n)$.
\item[d)] $|\zeta_i| \leq c < 1$ for all $i = 1, \dots, n$.
\item[e)] The following limits hold in probability:
$$E[W^+W^-] = \lim_{n \to \infty} \frac{1}{n} \sum_{i=1} W_i^- W_i^+ \qquad \text{and} \qquad E[W^+|Q|] = \lim_{n \to \infty} \frac{1}{n} \sum_{i=1} W_i^+ |Q_i|,$$
with $E[ W^+ W^- + W^+ |Q|] < \infty$. 
\item[f)] The vectors $(W^-, Q)$ and $(W^+, \zeta)$ are independent.
\end{enumerate}
\end{assumption}

\begin{remark}
Note that Assumption~\ref{A.ExtendedTypes} (a)-(c) implies Assumption~\ref{A.Types} (a)-(c).
\end{remark}

Our main result on the distribution of the rank of a randomly chosen vertex in the inhomogeneous random digraph from Section~\ref{S.Model} is given below. To avoid repetition, we refer the reader to \cite{chenetal} or \cite{Jel_Olv_12} for a detailed description of the attracting endogenous solution $\mathcal{R}$ to \eqref{eq:SFPE}, as well as its asymptotic behavior in terms of that of $\mathcal{N}, \mathcal{Q}, \mathcal{C}$; $\Rightarrow$ denotes weak convergence.

\begin{theo} \label{T.MainPageRank}
Suppose that Assumption~\ref{A.ExtendedTypes} holds,  and let $R_\xi$ denote the rank of a uniformly chosen vertex in the inhomogeneous random digraph $\mathcal{G}(V_n, E_n)$. Then, as $n \to \infty$, 
\begin{equation} \label{eq:LinearCombination}
R_\xi \Rightarrow \mathcal{R},
\end{equation}
where $\mathcal{R}$ is the attracting endogenous solution to \eqref{eq:SFPE}. The distributions of all the random variables involved in \eqref{eq:SFPE} are given below:
\begin{align*}
P(\mathcal{N} = m, \mathcal{Q} \in dq) &= E\left[1(Q \in dq) \cdot \frac{e^{-E[W^+] W^-/\theta} (E[W^+] W^-/\theta)^m}{m!} \right], \quad m = 0,1, \dots, \\
P(\mathcal{C}_1 \in dt) &= \frac{E[ 1(\zeta/(Z^++1) \in dt ) W^+]}{E[W^+]},
\end{align*}
and $Z^+$ is a mixed Poisson random variable with parameter $E[W^-] W^+/\theta$.
\end{theo}

The proof of Theorem~\ref{T.MainPageRank} is based on a coupling argument between a graph exploration process and a multi-type branching process, which is  similar to the techniques used in \cite{chenetal} for the analysis of generalized PageRank on the directed configuration model. Together with the results in \cite{chenetal} and Remark~\ref{R.PowerLaw}, Theorem~\ref{T.MainPageRank} provides further evidence of the ``universality" of the power-law hypothesis on scale-free directed complex networks. Some numerical examples illustrating the convergence of $R_\xi \Rightarrow \mathcal{R}$ for all the models in Example~\ref{E.IRG} are included in \ref{Appx.Numerical}.

In the following section we explain the main steps involved in the proof of Theorem~\ref{T.MainPageRank}, postponing all the technical proofs to Section~\ref{S.Proofs}.

\subsection{Deriving the SFPE approximation}

To make the proof of Theorem~\ref{T.MainPageRank} easier to follow, we have divided it into three main steps: 1) approximating the rank using the local neighborhood, 2) coupling with a branching process, and 3) proving convergence to the attracting endogenous solution.

\subsubsection{Approximating the rank using the local neighborhood} \label{SSS.LocalNeighborhood}

The first step towards proving Theorem~\ref{T.MainPageRank} consists in showing that it is enough to consider only the local neighborhood of each vertex in the graph to compute its rank. The first observation we make  is that the system of linear equations given by \eqref{eq:scaleFreePageRank} can be written in matrix notation as
$${\bf R} = {\bf R} {\bf M} + {\bf Q},$$
where ${\bf R} = (R_1, \dots, R_n)$, ${\bf Q} = (Q_1, \dots, Q_n)$ and the matrix ${\bf M}$ has $(i,j)$th component
$$M_{ij} = s_{ij} C_i,$$
where $s_{ij}$ is the number of edges from $i$ to $j$. Recall that $C_j = \zeta_j/D_j^+$, where $D_j^+$ is the out-degree of vertex $j$ and $|\zeta_j| \leq c < 1$ for all $j \geq 1$. Since the graphs we consider here are simple, we have $s_{ij} \in \{0, 1\}$, however, the definition of matrix {\bf M} also applies to multigraphs. It follows that the rank vector ${\bf R}$ can be written as
$${\bf R} = {\bf R}^{(n,\infty)} = \sum_{i=0}^\infty {\bf Q} {\bf M}^i.$$
Next, define $(R_1^{(n,k)}, \dots, R_n^{(n,k)}) = {\bf R}^{(n,k)} = \sum_{i=0}^k {\bf Q} {\bf M}^i$, and note that the i.i.d.~nature of the type sequence implies that all the coordinates of the vector ${\bf R}^{(n,\infty)} - {\bf R}^{(n,k)}$  are identically distributed (they are not identically distributed given $\mathscr{F}_n$). It follows from the exact arguments used in Section~4.2 in \cite{chenetal} that for a randomly chosen vertex $\xi$, 
\begin{equation} \label{eq:Local}
\mathbb{P}_n\left( \left| R_\xi^{(n,\infty)} - R_\xi^{(n,k)} \right| > x^{-1} \right) \leq \frac{x c^k}{1-c}  \cdot \frac{1}{n} \sum_{i=1}^n |Q_i|
\end{equation}
for any $x \geq 1$. 

Note that the calculation of each of the $R_i^{(n,k)}$, $i = 1, \dots, n$, requires only information about the vertices in the graph having a directed path to vertex $i$ of length at most $k$, i.e., it can be computed using only the local (inbound) neighborhood of each vertex.

\subsubsection{Coupling with a branching process} \label{SS.Coupling}

Now that we have reduced the problem of analyzing a randomly chosen component of the vector ${\bf R}^{(n,\infty)}$ to that of analyzing the corresponding component of the vector ${\bf R}^{(n,k)}$, the next step is to couple $R_\xi^{(n,k)}$ with the rank of the root node of a branching process. For the directed configuration model analyzed in \cite{chenetal}, the coupling was done with a marked Galton-Watson process, referred to as a ``thorny branching process'' in \cite{chenetal},  that was then used to define a weighted branching process \cite{rosler1}. The same idea works also for the inhomogeneous random digraphs considered here, although the coupling is more easily understood if instead of using from the beginning a marked Galton-Watson process we first consider a marked multi-type branching process. The marks include the number of outbound neighbors, the damping factor and the personalization value of each vertex discovered during the graph exploration process.

As it is usual when analyzing trees, we index the nodes with a label that allows us to trace their entire path from the root. More precisely, denote the root node $\emptyset$, and label its offspring as $\{1, 2, \dots, \hat N_\emptyset\}$, where $\hat N_\emptyset$ is the number of offspring that $\emptyset$ has. Set $\hat A_0 = \{\emptyset\}$ and $\hat A_1 = \{ 1, 2, \dots, \hat N_\emptyset\}$ to be the sets of individuals in generation zero and generation one of the tree, respectively. In general, we use $\hat A_k$ to denote the set of individuals in the $k$th generation of the tree, and a node/individual in $\hat A_k$ has a label of the form ${\bf i} = (i_1, \dots, i_k) \in \mathbb{N}_+^k$. Moreover, the set $\hat A_{k+1}$ can be constructed recursively according to 
$$\hat A_{k+1} = \{ ({\bf i}, j): {\bf i} \in \hat A_k, \, 1 \leq j \leq \hat N_{\bf i} \},$$
where $\hat N_{\bf i}$ is the number of offspring of node ${\bf i}$, and we use $({\bf i}, j) = (i_1, \dots, i_k, j)$ to denote the index concatenation operation; if ${\bf i} = \emptyset$, then $({\bf i}, j) = j$.  We use throughout the paper $\mathcal{U} = \bigcup_{k = 0}^\infty \mathbb{N}_+^k$, with the convention that $\mathbb{N}_+^0 = \{ \emptyset \}$.

To describe the multi-type branching process used in the coupling, we assume that each node in the tree has a type from the set $\mathscr{W}_n = \{ {\bf W}_i: 1 \leq i \leq n\}$, where ${\bf W}_i = (W_i^-, W_i^+, Q_i, \zeta_i)$.  Individuals in the tree have a random number of offspring, potentially of various types, independently of all other nodes. More precisely, if we let $Z_{ji}$ denote the number of offspring of type ${\bf W}_j$ that an individual of type ${\bf W}_i$ has, we have that for $(m_1, \dots, m_n) \in \mathbb{N}^n$,
\begin{equation} \label{eq:MultiOffspring}
\mathbb{P}_{n} \left( Z_{1i} = m_1, \dots, Z_{ni} = m_n \right) =  \prod_{j=1}^n \frac{e^{-q_{ji}^{(n)}} (q_{ji}^{(n)})^{m_j}}{m_j!}, 
\end{equation}
where 
$$q_{ji}^{(n)} =  \frac{  (W_j^+ \wedge a_n) (W_i^- \wedge b_n)}{\theta n}, \qquad 1\leq i,j \leq n,$$
and $a_n, b_n \geq 1$ are sequences to be determined later.
To simplify the notation, we write ${\bar W}_i^+ = W_i^+ \wedge a_n$ and ${\bar W}_i^- = W_i^- \wedge b_n$. Note that the random variables $\{ Z_{ji}: 1 \leq j \leq n\}$ are conditionally independent (given $\mathscr{F}_n$) Poisson random variables with the mean of $Z_{ji}$ equal to $q_{ji}^{(n)}$. To avoid the label of a node from giving us any information about its type, we assume that all $\hat N_{\bf i}$ offspring of node ${\bf i}$ are permuted uniformly at random before being assigned a label of the form $({\bf i}, j)$, $j = 1, \dots, \hat N_{\bf i}$. 

To make this a marked multi-type branching process, we give to each node ${\bf i}$ in the tree a mark $\hat D_{\bf i}$, such that if $\bf i$ has type ${\bf W}_s$, then
\begin{align} \label{eq:MarkDistr}
\mathbb{P}_{n} \left( \left. \hat D_{\bf i}  = m \right| {\bf i} \text{ has type } {\bf W}_s \right) = \frac{e^{- {\bar W}_s^+ \Lambda_n^-/(\theta n)} ( \bar{W}_s^+ \Lambda_n^-/(\theta n))^m}{m!}, \qquad m = 0, 1, 2, \dots,
\end{align}
independently of all other nodes. Here and in the sequel, $\Lambda_n^- = \sum_{i=1}^n {\bar W}_i^- $ and $\Lambda_n^+ = \sum_{i=1}^n {\bar W}_i^+$.  We refer to this marked multi-type branching process as a Poisson branching tree (PBT).

As mentioned earlier, it turns out that the PBT we just described can also be thought of as a marked Galton-Watson process.  To see this, note that the properties of the Poisson distribution imply that the type of a node ${\bf i}$ in the tree is independent of the type of its parent, as the following result shows (its proof is given in Section~\ref{SS.CouplingProofs}). 

\begin{lemma} \label{L.Independence}
For any node ${\bf i}$ in the PBT and any $1 \leq r,s \leq n$, we have
$$\mathbb{P}_{n} ( \text{${\bf i}$ has type ${\bf W}_s$} | \text{parent has type ${\bf W}_r$}) = \frac{{\bar W}_s^+}{\Lambda_n^+}.$$
\end{lemma}

This means that we could construct the PBT by assigning to each node {\bf i} in the tree a number of offspring $\hat N_{\bf i}$ and then sampling their types according to Lemma~\ref{L.Independence}, independently of everything else. The marks $\hat D_{({\bf i},j)}$ of each of these offspring would then be sampled according to \eqref{eq:MarkDistr}. Since the type of the root node is chosen uniformly at random from the set $\mathscr{W}_n = \{ {\bf W}_i: 1 \leq i \leq n\}$, the distribution of $\hat N_\emptyset$ may be different from that of all other nodes. This effect will disappear in the limit due to Assumption~\ref{A.ExtendedTypes}(f).
 
We now explain how to construct a coupling of the inhomogeneous random digraph $\mathcal{G}(V_n, E_n)$ and a PBT.  We start by choosing uniformly at random a vertex in the graph, call it $\xi$, and then exploring its in-component using a breadth-first exploration process. The coupled PBT is constructed to be in perfect agreement with the graph exploration process for a number of generations large enough to ensure that the rank of the randomly chosen node can be accurately approximated by its rank computed up to that point.  The exploration process will have even and odd steps: in Step $2k-1$ we will discover the set of vertices that have a directed path of length $k$ to the randomly chosen vertex, in Step $2k$ we will uncover all the outbound neighbors of the vertices discovered in Step $2k-1$. To keep track of this process, each vertex in the graph exploration will be assigned one of three labels: \{active, inactive, dead\}; vertices that have not been uncovered have no label. Active vertices will be those that are currently most distant from the randomly chosen vertex, and all we know about them is that they have an outbound edge connecting them to the in-component of the first vertex.  The vertices that have already been added to the exploration process, and whose inbound neighbors have been discovered, will be labeled dead. Vertices that have been discovered as additional outbound neighbors of active vertices are labeled inactive, and all we know about them is that they have an inbound edge connecting them to a vertex in the in-component we are exploring. Figure~\ref{F.Exploration} illustrates this process. 

\begin{figure}[t]
\centering
\includegraphics[scale = 0.6]{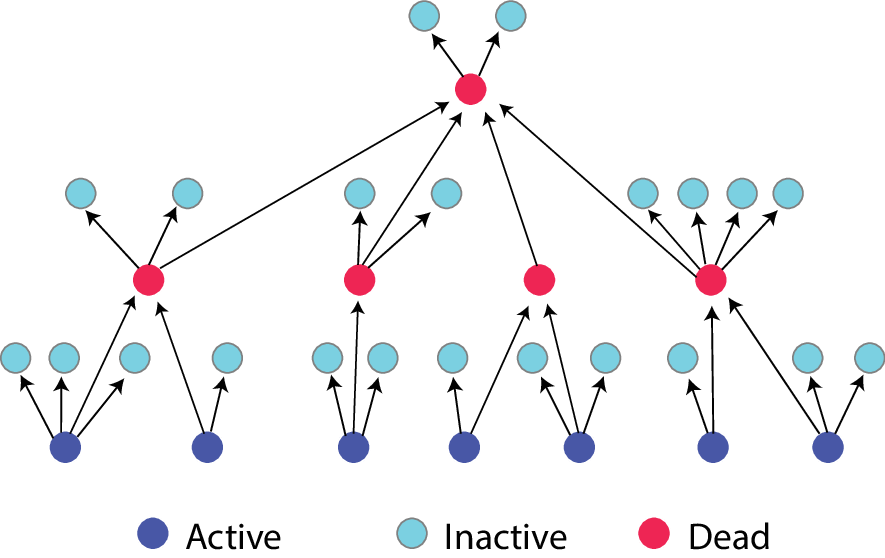}
\caption{Graph exploration process after completing Step 4.} \label{F.Exploration}
\end{figure}

For the coupled PBT we will need to keep track of the active vertices at the end of Step $2k-1$ of the graph exploration process, which will constitute the $k$th generation of nodes/individuals in the tree. We will also keep track of the inactive vertices by defining a similar set composed of all the types sampled during the creation of the marks $\{ \hat D_{\bf i}\}$ (note that in the graph each type appears only once, since each type can be identified with one of the $n$ vertices, while on the tree types can appear repeatedly). The notation below will help us with the construction of the coupling. 

For $k = 1, 2, \dots$, let
\begin{align*}
A_k &= \text{set of ``active'' vertices at the end of Step $2k-1$.} \\
T_k &= \bigcup_{m=0}^{k-1} A_m =  \text{set of ``dead'' vertices after Step $2k-1$}; \, T_0 = \varnothing. \\
I_k &= \text{set of ``inactive'' vertices after Step $2k$.} \\
\hat A_k &= \text{set of nodes in the PBT at distance $k$ from the root}. \\
\hat T_k &= \bigcup_{m=0}^{k-1} \hat A_m =  \text{set of nodes in the PBT at distance at most $k-1$ from the root}; \, \hat T_0 = \varnothing. \\
\hat I_k &= \text{set of ``inactive'' types in the PBT belonging to nodes at distance at most $k$ from } \\
&\hspace{5mm} \text{the root}.
\end{align*}

{\bf Note:} The sets $T_k$ and $\hat T_k$ grow in odd steps of the exploration process, while the sets $I_k$ and $\hat I_k$ do so in even steps. 

We now describe the coupling, for which we will require a sequence $\{ U_{ij}: 1 \leq i, j \leq n\}$ of i.i.d.~Uniform$(0,1)$ random variables that will be the same for the graph exploration process and the construction of the PBT.  To make the role that the choice of the first node plays in the coupling explicit, we state our coupling results in terms of the first node we choose. The coupling has odd steps and even steps, during odd steps we discover new nodes in the inbound component of the first node, in even steps we explore the outbound neighbors (which will become the marks) of the nodes discovered in the previous step. Throughout the paper we will use  $g^{-1}(u) = \inf\{ x\in \mathbb{R}: g(x) \geq u \}$ to denote the generalized inverse of function $g$ and $|A|$ to denote the cardinality of set $A$.

{\em Construction of the graph:}

Step 0: Choose the vertex whose neighborhood will be explored, say vertex $i$. Set $A_0 = \{i\}$ and label vertex $i$ as ``active''. To reveal all its outbound edges realize  $X_{it} = 1(U_{it} > 1-p_{it}^{(n)})$, $t = 1, 2, \dots, n$, $t \neq i$. If $X_{it} = 1$, label node $t$ as ``inactive", so $I_0 = \{ t \in \{1, 2, \dots, n\} \setminus \{i\}: X_{it} = 1 \}$.

In Step $2k-1$, $k \geq 1$, we explore the inbound neighbors of nodes in the set $A_{k-1}$. For each $i \in A_{k-1}$:
\begin{itemize}
\item[1)] For all $j =1, 2, \dots, n$, $j \neq i$ and $j \notin T_{k-1}$:
\begin{itemize}
\item[i.] Realize $X_{ji} = 1(U_{ji} > 1- p_{ji}^{(n)})$. 
\item[ii.] If $X_{ji} = 1$ and node $j$ was previously labeled ``inactive'', relabel it as ``active''.
\end{itemize}
\item[2)] Label $i$ as ``dead".
\end{itemize}

In Step $2k$, $k \geq 1$, we explore the outbound neighbors of all the nodes in the set $A_{k}$. For each $j \in A_{k}$ and all $t = 1, 2, \dots, n$, $t \neq j$ and $t \notin T_k$:
\begin{itemize}
\item[1)] Realize $X_{jt} = 1(U_{jt} > 1- p_{jt}^{(n)})$. 
\item[2)] If $X_{jt} = 1$ label node $t$ as ``inactive'' (if it was already ``active" it will have two labels). 
\end{itemize}
Step $2k$ ends when we have uncovered all the nodes in $A_k$ as well as their outbound neighbors.  

{\em Coupled construction of the PBT:}

To each node ${\bf i}$ in the tree we will also determine its {\em mark} $\hat D_{\bf i}$ (the type of vertex {\bf i} includes the values of $Q_{\bf i}$ and ${\bf \zeta}_{\bf i}$ so we can ignore those in the coupling).  This value $\hat D_{\bf i}$ will be created independently for each node in the PBT according to~\eqref{eq:MarkDistr}, but will be coupled with the creation of ``inactive'' vertices the first time that a type appears. As long as the coupling holds, we choose nodes in the tree in the same order as in the graph,   thus $\hat D_{\bf i}$ represents the out-degree of the corresponding node in the graph.

Step 0: Set $\hat A_0 = \{ \emptyset\}$ and set the root of the PBT, $\emptyset$, to have type ${\bf W}_i$, where $i$ is the vertex chosen in Step 0 of the graph construction. Define $G_{ij}(x) = \sum_{t=0}^{\lfloor x \rfloor} e^{-q_{ij}^{(n)}} (q_{ij}^{(n)})^t /t!$ to be the distribution of $Z_{ij}$, where $Z_{ij}$ has the interpretation of being the number of offspring of type ${\bf W}_i$ that a node of type ${\bf W}_j$ has. Next, realize all the $Z_{it} = G_{it}^{-1}(U_{it})$ for $t = 1, 2, \dots, n$, $t \neq i$, and $Z_{ii}^* \sim$~Poisson$(q_{ii}^{(n)})$, independent of $U_{it}$ and of any other $Z_{it}$, $t\neq i$. Set 
$$\hat D_\emptyset = Z_{ii}^* + \sum_{1 \leq t \leq n, \, t \neq i} Z_{it} $$
and for each $Z_{it} \geq 1$ (or $Z_{ii}^* \geq 1$) add type ${\bf W}_t$ (or ${\bf W}_i$) to $\hat I_0$.

In Step $2k-1$, $k \geq 1$, we identify the individuals, and their types, in the $k$th generation of the PBT. For each node ${\bf i} \in \hat A_{k-1}$:
\begin{itemize}
\item[a)] If node ${\bf i}$ is the first node in the PBT to have  type ${\bf W}_i$ proceed as follows:
\begin{itemize}
\item[1)] For $j = 1, 2, \dots, n$, $j \neq i$:
\begin{itemize}
\item[i.] Realize $Z_{ji} = G_{ji}^{-1}(U_{ji})$. 
\item[ii.] If $Z_{ji} \geq 1$, add $Z_{ji}$ nodes of type ${\bf W}_j$ to the active set.
\end{itemize}
\item[2)] Realize $Z_{ii}^*\sim$Poisson$(q_{ii}^{(n)})$, independently of anything else, and assign a number $Z_{ii}^*$ of type ${\bf W}_i$ offspring. 
\end{itemize}
\item[b)] If node ${\bf i}$ is not the first node in the PBT to have type ${\bf W}_i$, sample a vector $(V_1, V_2, \dots, V_n)$ of i.i.d.~Uniform$(0,1)$ random variables, independent of the sequence $\{ U_{ij}: 1 \leq i,j \leq n\}$, and of any other $V_i$'s sampled before, and assign to node ${\bf i}$ a number $G_{ji}^{-1}(V_j)$ of type ${\bf W}_j $ offspring, for $j = 1, 2, \dots, n$. \end{itemize}
In Step $2k$, $k \geq 1$, we sample the marks of all the nodes in $\hat A_k$. For each node ${\bf i} \in \hat A_k$:
\begin{itemize}
\item[a)] If node ${\bf i}$ is the first node in the PBT to have type ${\bf W}_j$ proceed as follows:
\begin{itemize}
\item[i.] Realize all the $Z_{jt} = G_{jt}^{-1}(U_{jt})$ for $t = 1,2, \dots, n$, and $t \notin T_k$.
\item[ii.] Sample $Z_{jt}^*\sim$Poisson($q_{jt}^{(n)}$) for $t$ already ``dead", independently of everything else.
\item[iii.] Set
$$\hat D_{\bf i}  = \sum_{1 \leq t \leq n, \, t \neq j,i, t \notin T_k} Z_{jt} + Z_{jj}^* + Z_{ji}^* + \sum_{t \in T_k} Z_{jt}^*,$$
and add all the corresponding types (i.e., add type ${\bf W}_t$ if $Z_{jt} \geq 1$ or $Z_{jt}^* \geq 1$) to the ``inactive'' set. 
\end{itemize}
\item[b)] If node ${\bf i}$ is not the first node in the PBT to have type ${\bf W}_j$, sample a vector $(V_1, V_2, \dots, V_n)$ of i.i.d.~Uniform(0,1) random variables, independent of the sequence $\{ U_{ij}: 1 \leq i,j \leq n\}$, and of any other $V_i$'s sampled before, and assign to node ${\bf i}$ a number $G_{ji}^{-1}(V_j)$ of type ${\bf W}_j $ offspring, for $j = 1, 2, \dots, n$. Set $\hat D_{\bf i} = \sum_{t = 1}^n Z_{jt}^*$ and add the corresponding types to the ``inactive" set.
\end{itemize}

\begin{defn} \label{D.CouplingBreaks}
We say that the coupling of the graph and the PBT holds up to Step $2k$ if the graph exploration process up to a distance $k$ from the first (root) vertex is identical to that of the PBT, i.e., $|A_m| = |\hat A_m|$ and $I_m = \hat I_m$ for all $0 \leq m \leq k$.  Let $\tau$ be the step in the graph exploration process during which the coupling breaks.
\end{defn}

Before stating the main result obtained from this step, we need to define:
\begin{equation} \label{eq:DeltaDef}
\Delta_n = d_1( F_{n}, F),
\end{equation}
where $d_1(F,G)$ is the Kantorovich-Rubinstein distance (or Wasserstein distance of order one) between distributions $F$ and $G$. In particular,
$$d_1(F,G) = \inf_{{\bf X} \sim F,{\bf Y} \sim G} E[ \norm{ {\bf X}  - {\bf Y} }_1],$$
where the infimum is taken over all possible couplings of ${\bf X}$ and ${\bf Y}$, where ${\bf X}$ has distribution $F$ and ${\bf Y}$ has distribution $G$. Since convergence in $d_1$ is equivalent to weak convergence and convergence of the first absolute moments (see Theorem~6.9 and Definition~6.8(i) in \cite{Villani_2009}), Assumption~\ref{A.Types} (a)-(b) implies that 
$$\Delta_n \xrightarrow{P} 0  \qquad \text{as } n \to \infty.$$
Note that $\Delta_n$ is measuring the distance between the empirical distribution of the extended types and its limiting distribution, while $\mathcal{E}_n$ is measuring the error between the edge probabilities in the graph and their asymptotic equivalents.

Let $\mathbb{P}_{n,i}(\cdot) = \mathbb{P}_n(\cdot | A_0 = \{ i\})$ or $\mathbb{P}_{n,i}(\cdot ) = \mathbb{P}_n(\cdot | \emptyset \text{ has type ${\bf W}_i$})$, depending on whether the event refers to the graph or to the PBT, respectively.  The main result obtained from this step is given below, and its proof is given in Section~\ref{SS.CouplingProofs}. The theorem provides an explicit upper bound for the probability that the coupling breaks before step $2k$; we will later choose the sequences $a_n, b_n, c_n, s_n$ in such a way that the bound converges to zero as $n \to \infty$.

\begin{theo} \label{T.CouplingBreaks}
Fix $a_n, b_n, c_n, s_n \geq 1$ such that $s_n \leq a_n \wedge b_n$. Then, for any $k \in \mathbb{N}_+$, 
\begin{align*}
\mathbb{P}_{n,i} (\tau \leq 2k) &\leq  1(W_i^+ > a_n) + \mathcal{P}_n^+(i) +  (\mathcal{H}_n/\theta) \bar W_i^+ (\Delta_n + g_-(b_n) + a_n b_n/n)  \\
&\hspace{5mm} +  \mathcal{H}_n  k s_n \left(  g_+(c_n) + c_n \left( \mathcal{E}_n + \Delta_n + g_+(a_n) + g_-(b_n) + a_n b_n/n \right)  \right)  \\
&\hspace{5mm} + \mathbb{P}_{n,i}( |\hat T_{k+1}| \vee |\hat I_k| > s_n),
\end{align*}
where $\mathcal{P}_n^+(i) = \sum_{1 \leq j \leq n, j \neq i} |p_{ij}^{(n)} -(r_{ij}^{(n)} \wedge 1)| $, $g_-(x) = E[(W^- - x)^+]$, $g_+(x) =E[(W^+ - x)^+]$, and 
$$\mathcal{H}_n = \frac{2 n}{\Lambda_n^+}  (1+ \Delta_n/\theta)^2(3 + a_n^{-1} + b_n^{-1} + c_n^{-1} ).$$ 
\end{theo}

\subsubsection{Convergence to the attracting endogenous solution}
 
In view of Theorem~\ref{T.CouplingBreaks}, computing $R_\xi^{(n,k)}$ requires us to analyze only the first $k$ generations of the PBT, provided $\tau > 2k$. In order to do so we first explain how to use the marks $\{ \hat D_{\bf i}\}$ to compute the generalized PageRank of the root node of the PBT. For each node ${\bf i}$ in the PBT having type ${\bf W}_s$, we define its weight and personalization value according to
$$\hat C_{\bf i} = \frac{\zeta_s}{\hat D_{\bf i} + 1} \qquad \text{and} \qquad \hat Q_{\bf i} = Q_s.$$
Using the tree-indexing notation introduced in Section~\ref{SS.Coupling}, we iteratively compute the rank of the root node of the PBT, denoted $\hat R_\emptyset^{(n,k)}$, according to
\begin{equation} \label{eq:PRrecursion}
\hat R_{\bf i}^{(n,k)} = \sum_{j=1}^{\hat N_{\bf i}} \hat C_{({\bf i},j)} \hat R_{({\bf i},j)}^{(n,k-1)} + \hat Q_{\bf i}, \qquad k \geq 1, \qquad \hat R_{\bf j}^{(n,0)} = 0,
\end{equation}
where $\hat N_{\bf i}$ is the total number of offspring that node ${\bf i}$ has. In view of Lemma~\ref{L.Independence} and the observation that the type of the root node will be chosen uniformly at random, we have that the distribution of $(\hat N_\emptyset, \hat Q_\emptyset)$ is given by
\begin{equation} \label{eq:RootVector}
\mathbb{P}_{n}\left( \hat N_\emptyset = m, \, \hat Q_\emptyset = q \right) = \sum_{s=1}^n 1(Q_s = q) \cdot \frac{e^{-\frac{\Lambda_n^+}{\theta n} {\bar W}_s^- } (\Lambda_n^+ {\bar W}_s^-/(\theta n))^m}{m!} \cdot \frac{1}{n},
\end{equation}
for $m \in \mathbb{N}$ and $q \in \mathbb{R}$. Moreover, for any node ${\bf i} \neq \emptyset$, we have that
\begin{align}
&\mathbb{P}_{n}\left( \hat N_{\bf i} = m, \, \hat Q_{\bf i} = q, \, \hat C_{\bf i} = t \right) \notag \\
&= \sum_{s=1}^n \mathbb{P}_{n}\left( \left. \hat N_{\bf i} = m, \, \hat Q_{\bf i} = q, \, \hat C_{\bf i} = t \right| {\bf i} \text{ has type } {\bf W}_s \right) \frac{{\bar W}_s^+}{\Lambda_n^+}  \notag \\
&= \sum_{s=1}^n 1(Q_s = q) \cdot \frac{e^{-\frac{\Lambda_n^+}{\theta n} {\bar W}_s^-} (\Lambda_n^+ {\bar W}_s^-/(\theta n))^m}{m!} \cdot \mathbb{P}_{n} (\zeta_s/(\hat D_{\bf i} +1) =t | {\bf i} \text{ has type ${\bf W}_s$}) \cdot \frac{{\bar W}_s^+}{\Lambda_n^+} \notag \\
&= \sum_{s=1}^n 1(Q_s = q, \zeta_s/t -1 \in \mathbb{N}) \cdot p(m; \Lambda_n^+ {\bar W}_s^-/(\theta n)) \cdot p(\zeta_s/t-1; \Lambda_n^- {\bar W}_s^+/(\theta n))\cdot \frac{{\bar W}_s^+ }{\Lambda_n^+}, \label{eq:TypicalVector}
\end{align}
for $m \in \mathbb{N}$ and $t,q \in \mathbb{R}$, where $p(m; \lambda) = e^{-\lambda} \lambda^m/m!$. Note that the independence of the edges implies that the sequence $\{(\hat N_{\bf i}, \hat Q_{\bf i}, \hat C_{\bf i}): {\bf i} \in \mathcal{U} \}$ consists of conditionally independent vectors given $\mathscr{F}_n$.  

Now that we have explained how to compute generalized PageRank on the PBT, we obtain, as a consequence of  Theorem~\ref{T.CouplingBreaks}, the following result for $R_\xi^{(n,k)}$; its proof is given in Section~\ref{SS.CouplingCorollary}. 

\begin{theo} \label{T.CouplingFinal}
Let $\xi$ be the index of a uniformly chosen vertex in $\mathcal{G}(V_n, E_n)$. Under Assumption~\ref{A.ExtendedTypes} (a)-(c) we have that for any fixed $k \in \mathbb{N}_+$,
\begin{align*}
\mathbb{P}_n\left( R_\xi^{(n,k)} \neq \hat R_\emptyset^{(n,k)}  \right) &\leq  \frac{1}{n} \sum_{i=1}^n \mathbb{P}_{n,i}(\tau \leq 2k)  \xrightarrow{P} 0,
\end{align*}
as $n \to \infty$.
\end{theo}

To make the connection with the SFPE, note that since we assume that $(W^-, Q)$ is independent of $(W^+, \zeta)$, the vectors $\{ (\hat N_{\bf i}, \hat Q_{\bf i}, \{ \hat C_{({\bf i},j)} \}_{j \geq 1} ): {\bf i} \in \mathcal{U} \}$ will be asymptotically independent, and therefore can be used to define a weighted branching process (WBP) with generic branching vector $(\mathcal{N}, \mathcal{Q}, \{ \mathcal{C}_j\}_{j \geq 1})$, where the latter is the distributional limit of $(\hat N_{\bf i}, \hat Q_{\bf i}, \{ \hat C_{({\bf i},j)} \}_{j \geq 1} )$, ${\bf i} \neq \emptyset$. Moreover, the $\{ \mathcal{C}_j\}_{j \geq 1}$ will be i.i.d.~and independent of $(\mathcal{N}, \mathcal{Q})$. We refer the reader to \cite{Jel_Olv_10, chenetal} for more details on the description and basic properties of WBPs of this form.  The proof of this convergence in the Kantorovich-Rubinstein metric (see, e.g., Chapter 6 in \cite{Villani_2009}) is given in Section~\ref{S.WBPconvergenceProofs}. Once this convergence is established, the convergence of $\hat R_\emptyset^{(n,k)}$ will follow from Theorem~2 in \cite{chenolvera1}. The precise statement of this last step in the proof of Theorem~\ref{T.MainPageRank} is given below.

\begin{theo} \label{T.SFPEConvergence}
Under Assumption~\ref{A.ExtendedTypes} (a)-(b) \& (d)-(f), we have that for any fixed $k \in \mathbb{N}_+$  the rank of the root node in the PBT computed up to generation $k$ satisfies
$$\hat R_\emptyset^{(n,k)} \Rightarrow \mathcal{R}^{(k)} \qquad \text{and} \qquad  \mathbb{E}_{n}\left[ |\hat R_\emptyset^{(n,k)}| \right] \stackrel{P}{\longrightarrow} E[|\mathcal{R}^{(k)}|] , \qquad n \to \infty,$$
where $\mathcal{R}^{(k)} \to \mathcal{R}$ a.s.~as $k \to \infty$, with $\mathcal{R}$ defined as in Theorem~\ref{T.MainPageRank}. 
\end{theo}

The proof of Theorem~\ref{T.MainPageRank} is obtained by combining \eqref{eq:Local}, Theorem~\ref{T.CouplingFinal}, and Theorem~\ref{T.SFPEConvergence}. All the proofs are given in Section~\ref{S.Proofs}.

\section{Proofs} \label{S.Proofs}

This section includes the proofs of Theorem~\ref{T.MixedPoisson}, Lemma~\ref{L.Independence}, Theorem~\ref{T.CouplingBreaks}, Theorem~\ref{T.CouplingFinal}, Theorem~\ref{T.SFPEConvergence}, and ends with the proof of Theorem~\ref{T.MainPageRank}. Since some of the proofs are rather technical and require some preliminary results, we have organized them in subsections.  We start with the proof of Theorem~\ref{T.CouplingBreaks} followed by that of Theorem~\ref{T.CouplingFinal}, since their proofs can be used to give a short proof of Theorem~\ref{T.MixedPoisson}.

\subsection{Proofs of Lemma~\ref{L.Independence} and Theorem~\ref{T.CouplingBreaks}} \label{SS.CouplingProofs} 

The proof of Theorem~\ref{T.CouplingBreaks} is rather long, so we split some of the technical steps into three preliminary results to ease its reading.  We point out that all of the results in this section are proven conditionally on the type sequence $\mathscr{W}_n = \{ {\bf W}_i: 1 \leq i \leq n\}$, and therefore, all the expectations that appear throughout the section are finite. In some of our results related to the coupling, we use the notation $\mathbb{P}_{n,i}(\cdot ) = \mathbb{P}_n(\cdot | A_0 = \{ i\})$ or $\mathbb{P}_{n,i}(\cdot ) = \mathbb{P}_n(\cdot | \emptyset \text{ has type ${\bf W}_i$})$, depending on whether the event occurs on the graph or on the PBT, respectively. Similarly, we use $\mathbb{E}_{n,i}[ \cdot ] = \mathbb{E}_{n}[ \cdot | A_0 = \{i \}]$ or $\mathbb{E}_{n,i}[ \cdot ] = \mathbb{E}_{n}[ \cdot | \emptyset \text{ has type } {\bf W}_i]$ to denote the corresponding conditional expectations. 

\begin{proof}[Proof of Lemma~\ref{L.Independence}]
We start by noting that
\begin{align*}
&\mathbb{P}_{n} ( \text{${\bf i}$ has type ${\bf W}_s$} | \text{parent has type ${\bf W}_r$}) \\
&= \mathbb{E}_{n} \left[ \left. \frac{Z_{sr}}{ Z_{1r} + \dots + Z_{nr} } \right|   Z_{1r} + \dots + Z_{nr} \geq 1 \right] ,
\end{align*}
where the $\{Z_{jr}: 1 \leq j \leq n\}$ are independent Poisson random variables with $\mathbb{E}_{n}[ Z_{jr}] = q_{jr}^{(n)}$. Since for two independent Poisson random variables $X$ and $Y$ with means $\mu$ and $\lambda$, respectively, we have that $X | X+Y = n$ has a Binomial$(n, \mu/(\mu+\lambda))$ distribution, then
\begin{align*}
E\left[ \left. \frac{X}{X+Y}  \right| X+Y \geq 1 \right] &= \frac{1}{P(X+Y \geq 1)} E\left[ \frac{X}{X+Y} \cdot 1(X+Y \geq 1) \right] \\
&= \frac{1}{P(X+Y \geq 1)} \sum_{n=1}^\infty \frac{1}{n} E[X | X+Y = n] \cdot \frac{e^{-\mu-\lambda} (\mu+\lambda)^n}{n!} \\
&= \frac{\mu}{(\mu+\lambda) P(X+Y \geq 1)} \sum_{n=1}^\infty \frac{e^{-\mu-\lambda} (\mu+\lambda)^n}{n!} = \frac{\mu}{\mu+\lambda}.
\end{align*}
It follows that
\begin{align*}
\mathbb{P}_{n} ( \text{${\bf i}$ has type ${\bf W}_s$} | \text{parent has type ${\bf W}_r$})  &= \frac{q_{sr}^{(n)}}{q_{1r}^{(n)} + \dots + q_{nr}^{(n)}} = \frac{{\bar W}_s^+}{\Lambda_n^+}. 
\end{align*}
\end{proof}

The proof of Theorem~\ref{T.CouplingBreaks} is divided into two parts, one that computes the probability that the coupling breaks in Step 0 and another that computes the probability that it breaks in Step $m$, $m \geq 1$. In both cases, the idea behind the proofs is to identify the possible ways in which the coupling can break in Step $m$, and carefully estimate their corresponding probabilities. To help explain the steps in the proofs that follow, it may be helpful to list the events that can lead to the coupling breaking in Step $m$. 

\begin{remark} \label{R.BreakCouplingReasons}
The coupling breaks at time $\tau$ for the following reasons:
\begin{itemize}
\item If $A_0 = \{i\}$, then $\tau = 0$ if: $Z_{it} \neq X_{it}$ for any $t = 1,2, \dots, n$, $t \neq i$, or $Z_{ii} \geq 1$.
\item $\tau = 2m-1$, $m \geq 1$, if for some $i \in A_{m-1}$ any of the following happen:
\begin{itemize}
\item[a.] $X_{ji} = 1$ for any $j \in I_{m-1}$ or any $j$ in the current active set when $i$ is explored (in which case a cycle is created).
\item[b.]  $Z_{ji} \geq 1$ for $j = i$ or $j \in I_{m-1}$.
\item[c.] $X_{ji} \neq Z_{ji}$ for some $j = 1, 2, \dots, n$, $j \neq i$, $j \notin I_{m-1}$. 
\end{itemize}
\item $\tau = 2m$, $m \geq 1$, if for some $j \in A_m$ either:
\begin{itemize}
\item[d.] $X_{jt} \neq Z_{jt}$ for some $t = 1, 2, \dots, n$, $t \neq j, i$, $t \notin T_m$.
\item[e.] $Z_{jt}^* \geq 1$ for $t = i$ or $t \in T_m$. 
\end{itemize}
\end{itemize}
\end{remark}

A first step in the derivation of the bounds we seek is the following preliminary result bounding the probabilities of having edge discrepancies between the exploration of the graph and of the coupled PBT, both inbound and outbound. Recall that $\Delta_n = d_1(F_n, F)$ was defined in \eqref{eq:DeltaDef}. 

\begin{lemma} \label{L.EdgeDiscrepancy}
For any $1 \leq i \leq n$ we have
\begin{align*}
 \mathbb{P}_n\left( \max_{1 \leq j \leq n, j \neq i} |X_{ji} - Z_{ji}| \geq 1 \right) &\leq \min\left\{ 1, 1(W_i^- > b_n) + \mathcal{P}_n^-(i) + \bar W_i^- \eta_n^- \right\}, \\
\mathbb{P}_n\left( \max_{1 \leq j \leq n, j \neq i} |X_{ij} - Z_{ij}| \geq 1 \right) &\leq \min\left\{ 1, 1(W_i^+ > a_n) + \mathcal{P}_n^+(i) +  \bar W_i^+ \eta_n^+\right\},
\end{align*}
where 
$$\mathcal{P}_n^-(i) =  \sum_{1 \leq j \leq n, j \neq i} \left| p_{ji}^{(n)} - (r_{ji}^{(n)} \wedge 1) \right| , \qquad \mathcal{P}_n^+(i) =  \sum_{1 \leq j \leq n, j \neq i} \left| p_{ij}^{(n)} - (r_{ij}^{(n)} \wedge 1) \right| , $$
$$\eta_n^- = ( \Delta_n + g_+(a_n) +  a_n b_n/n + a_n b_n \Delta_n/(\theta n))/\theta, \qquad \eta_n^+ = ( \Delta_n + g_-(b_n) + a_n b_n/n + a_n b_n \Delta_n/(\theta n))/\theta, $$
$$g_-(x) = E[(W^- - x)^+], \qquad \text{and} \qquad g_+(x) = E[(W^+ - x)^+].$$ 
\end{lemma}

\begin{proof}
The analysis of the two probabilities is essentially the same, so we only prove the result for outbound edges. Let $R_{ij} = 1(U_{ij} > 1- r_{ij}^{(n)})$ with $r_{ij}^{(n)} = W_i^+ W_j^-/(\theta n)$. The union bound gives:
\begin{align*}
\mathbb{P}_n\left( \max_{1 \leq j \leq n, j \neq i} |X_{ij} - Z_{ij}| \geq 1 \right) &\leq  1(W_i^+ > a_n) + 1(W_i^+ \leq a_n) \sum_{1 \leq j \leq n, j \neq i} \mathbb{P}_n( |X_{ij} - Z_{ij}| \geq 1).
\end{align*}
Now note that
\begin{align*}
\mathbb{P}_n( |X_{ij} - Z_{ij}| \geq 1) &= \mathbb{P}_n( |X_{ij} - Z_{ij}| \geq 1, |X_{ij} - R_{ij}| \geq 1) + \mathbb{P}_n( |X_{ij} - Z_{ij}| \geq 1, |X_{ij} - R_{ij}| = 0) \\
&\leq  \mathbb{P}_n( |X_{ij} - R_{ij}| \geq 1) + \mathbb{P}_n( |R_{ij} - Z_{ij}| \geq 1). 
\end{align*}
The first probability can be computed to be:
\begin{align*}
\mathbb{P}_n( |X_{ij} - R_{ij}| \geq 1) &= | p_{ij}^{(n)} - (r_{ij}^{(n)} \wedge 1)|. 
\end{align*}

To analyze each of probabilities involving $R_{ij}$ and $Z_{ij}$, note that 
\begin{align*}
\mathbb{P}_{n} \left( |R_{ij} - Z_{ij}| \geq 1 \right) &= \mathbb{P}_n(R_{ij} = 0, Z_{ij} \geq 1) + \mathbb{P}_n(R_{ij} = 1, Z_{ij} = 0) + \mathbb{P}_n(R_{ij} = 1, Z_{ij} \geq 2) \notag \\
&= \left( 1- (1 \wedge r_{ij}^{(n)}) - e^{-q_{ij}^{(n)}} \right)^+ +  \left( e^{-q_{ij}^{(n)}}  - 1 + (1 \wedge r_{ij}^{(n)})  \right)^+ \notag \\
&\hspace{5mm} + \min\left\{ 1 - e^{-q_{ij}^{(n)}} (1 + q_{ij}^{(n)}) , \, (1 \wedge r_{ij}^{(n)} )  \right\}  \notag\\
&= \left| 1- (1 \wedge r_{ij}^{(n)})  - e^{-q_{ij}^{(n)}} \right| + \min\left\{ (1 \wedge r_{ij}^{(n)}), \, e^{-q_{ij}^{(n)}} ( e^{q_{ij}^{(n)}} - 1 - q_{ij}^{(n)} ) \right\} .
\end{align*}
Now use the inequalities $e^{-x} \geq 1-x$, $e^{-x} - 1 + x \leq x^2/2$ and $e^x - 1 - x \leq x^2 e^x/2$ for $x \geq 0$, to obtain that
\begin{align*}
\mathbb{P}_{n} \left( |R_{ij} - Z_{ij}| \geq 1 \right) &\leq   r_{ij}^{(n)} - q_{ij}^{(n)} + \left| 1 - q_{ij}^{(n)} - e^{-q_{ij}^{(n)}} \right| + e^{-q_{ij}^{(n)}} ( e^{q_{ij}^{(n)}} - 1 - q_{ij}^{(n)} ) \\
&= r_{ij}^{(n)} - q_{ij}^{(n)} + e^{-q_{ij}^{(n)}} - 1 + q_{ij}^{(n)} +  e^{-q_{ij}^{(n)}} ( e^{q_{ij}^{(n)}} - 1 - q_{ij}^{(n)} ) \\
&\leq  r_{ij}^{(n)} - q_{ij}^{(n)} + (q_{ij}^{(n)})^2.
\end{align*}

It follows that
\begin{align*}
&1(W_i^+ \leq a_n) \sum_{1 \leq j \leq n, j \neq i} \mathbb{P}_n( |X_{ij} - Z_{ij}| \geq 1) \\
&\leq 1(W_i^+ \leq a_n)  \sum_{1 \leq j \leq n, j \neq i} \left( | p_{ij}^{(n)} - (r_{ij}^{(n)} \wedge 1)| + r_{ij}^{(n)} - q_{ij}^{(n)} + (q_{ij}^{(n)})^2   \right)  \\
&\leq \mathcal{P}_n^+(i) + \sum_{1 \leq j \leq n, j \neq i} \frac{\bar W_i^+ (W_j^- - {\bar W}_j^-)}{\theta n}  + \frac{({\bar W}_i^+)^2}{(\theta n)^2} \sum_{1 \leq j \leq n, j \neq i} ({\bar W}_j^-)^2 \\
&\leq \mathcal{P}_n^+(i) + \frac{\bar W_i^+}{\theta n} \sum_{j=1}^n (W_j^- - b_n)^+  + \frac{({\bar W}_i^+)^2 b_n \Lambda_n^-}{(\theta n)^2}.
\end{align*}
To further bound the second term note that if we let $(W^{(-,n)}, W^{(+,n)})$ denote a random vector distributed according to $F_{n}$ and $(W^-, W^+)$ a random vector distributed according to $F$, then
$$\frac{1}{n} \sum_{j=1}^n (W_j^- - b_n)^+ = \mathbb{E}_n\left[ (W^{(-,n)} - b_n)^+ \right] \leq d_1(F_{n}, F) + E\left[ (W^- - b_n)^+ \right] = \Delta_n + g_-(b_n).$$
And for the last term, 
$$\frac{({\bar W}_i^+)^2 b_n \Lambda_n^-}{(\theta n)^2} \leq \frac{\bar W_i^+ a_n b_n}{\theta^2 n} \cdot \mathbb{E}_n\left[ W^{(-,n)}  \right]  \leq \frac{\bar W_i^+ a_n b_n}{\theta^2 n} \left( \Delta_n + E[W^-] \right) .$$

We conclude that for $\eta_n^+$ as defined in the statement of the lemma, 
\begin{align*}
1(W_i^+ \leq a_n) \sum_{1 \leq j \leq n, j \neq i} \mathbb{P}_n( |X_{ij} - Z_{ij}| \geq 1) &\leq  \mathcal{P}_n^+(i) + \frac{\bar W_i^+}{\theta} ( \Delta_n + g_-(b_n)) + \frac{\bar W_i^+ a_n b_n}{\theta^2 n} (\Delta_n + E[W^-]) \\
&\leq \mathcal{P}_n^+(i) + \eta_n^+ \bar W_i^+,
\end{align*}
which in turn yields
\begin{align*}
\mathbb{P}_n\left( \max_{1 \leq j \leq n, j \neq i} |X_{ij} - Z_{ij}| \geq 1 \right) &\leq \min\left\{ 1, 1(W_i^+ > a_n) + \mathcal{P}_n^+(i) + \eta_n^+ \bar W_i^+ \right\}  .
\end{align*}
\end{proof}

We now give an upper bound for the probability that the coupling breaks on Step 0 when the starting vertex is $i$.

\begin{lemma} \label{L.TauZero}
We have
$$ \mathbb{P}_{n,i}(\tau = 0) \leq 1(W_i^+ > a_n)  +  \mathcal{P}_n^+(i) + \bar W_i^+ (\eta_n^+ + b_n/(\theta n)).$$
\end{lemma}

\begin{proof}
By the union bound followed by Lemma~\ref{L.EdgeDiscrepancy} we have,
\begin{align*}
\mathbb{P}_{n,i}(\tau = 0) &\leq \mathbb{P}_{n} \left( \max_{1 \leq t \leq n, t \neq i} |X_{it} - Z_{it}| > 0 \right) + \mathbb{P}_{n} \left( Z_{ii}^* \geq 1 \right) \\
&\leq 1(W_i^+ > a_n) + \mathcal{P}_n^+(i) + \bar W_i^+ \eta_n^+ + 1- e^{-q_{ii}^{(n)}}  \\
&\leq  1(W_i^+ > a_n)  +  \mathcal{P}_n^+(i) + \bar W_i^+ \eta_n^+ + q_{ii}^{(n)} \\
&\leq  1(W_i^+ > a_n)  +  \mathcal{P}_n^+(i) + \bar W_i^+ (\eta_n^+ + b_n/(\theta n)).
\end{align*}
\end{proof}

We now give an upper bound for the probability that the coupling breaks in Step $m$ for $m \geq 1$.

\begin{prop} \label{P.CouplingBreaksAtStepM}
Fix $c_n, s_n \geq 1$ with $s_n \leq a_n \wedge b_n$ and define for $m \geq 0$ the event $M_m =  \left\{ |\hat T_{m+1}| \vee |\hat I_m| \leq s_n \right\}$. Then,
for any $m \geq 1$, 
\begin{align*}
\mathbb{P}_{n,i}(\tau = 2m-1, M_{m})  &\leq  \frac{s_n}{\Lambda_n^+/n } \left( g_+(c_n) + \Delta_n + c_n \mathcal{E}_n + c_n \gamma_n^-    \right), \\
\mathbb{P}_{n,i}(\tau = 2m, M_m) &\leq \frac{s_n}{\Lambda_n^+/n} \left( g_+(c_n) + \Delta_n + c_n \mathcal{E}_n +  c_n \gamma_n^+    \right),
\end{align*}
where $g_+$ and $g_-$ are defined as in Lemma~\ref{L.EdgeDiscrepancy},
$$\mathcal{E}_n =  \frac{1}{n} \sum_{i=1}^n \sum_{1 \leq j \leq n, j \neq i} |p_{ij}^{(n)} - (r_{ij}^{(n)} \wedge 1)|,$$
and 
\begin{align*}
\gamma_n^- &= \frac{(E[W^-] + \Delta_n)}{\theta} (\Delta_n + g_+(a_n) +  3a_n b_n/n + a_n b_n \Delta_n/(\theta n) + a_n/n) , \\
\gamma_n^+ &= \frac{(E[W^+] + \Delta_n)}{\theta} (\Delta_n + g_-(b_n) + 2a_n b_n/n + a_n b_n \Delta_n/(\theta n)+ b_n/n).
\end{align*}
\end{prop}

\begin{proof}
We start by defining the following events:
\begin{align*}
F_{i}(I) &=  \left\{  \max_{1 \leq j \leq n, j \neq i} |X_{ji} - Z_{ji}|= 0, \, Z_{ii}+ \sum_{1 \leq j \leq n, \, j \neq i, \, j \in I}  Z_{ji}  = 0  \right\}, \\
G_{j}(D) &= \left\{ \max_{1 \leq t \leq n, t \neq j, \, t \in D^c} | X_{jt} - Z_{jt}| = 0, \, Z_{jj}^* + \sum_{1 \leq t \leq n, \, t\neq j, \, t \in D} Z_{jt}^* = 0 \right\}, \\
\mathbb{A}_i &= \{ 1 \leq j \leq n: j \text{ is active when the inbound neighbors of $i$ are explored} \} \\
H_k &= \bigcap_{i \in A_{k-1}} F_i(I_{k-1} \cup \mathbb{A}_i ), \\
J_k &= \bigcap_{j \in A_k} G_j(T_{k}). 
\end{align*}
Now note that for any $m \geq 1$, 
\begin{align*}
\mathbb{P}_{n,i}\left(\tau = 2m-1, M_{m} \right) &= \mathbb{P}_{n,i}\left( M_{m} \cap J_0 \cap  \bigcap_{k=1}^{m-1} \left( H_{k} \cap J_k  \right) \cap H_m^c \right) , \\
\mathbb{P}_{n,i}(\tau = 2m, M_{m}) &= \mathbb{P}_{n,i}\left(  M_{m} \cap J_0 \cap \bigcap_{k=1}^{m-1} \left( H_{k} \cap J_k  \right) \cap H_m \cap J_m^c \right),
\end{align*}
with the convention that $\bigcap_{k=1}^0 (H_k \cap J_k) = \Omega$. 

Let $\mathcal{F}_t$ denote the sigma-algebra that contains the history of the inbound exploration process in the graph as well as that of the PBT, up to the end of Step $t$ of the graph exploration process. It follows that we can write:
\begin{align*}
\mathbb{P}_{n,i}(\tau = 2m-1, M_{m}) &= \mathbb{E}_{n,i}\left[ 1\left( M_{m-1} \cap J_0 \cap  \bigcap_{k=1}^{m-1} \left( H_{k} \cap J_k  \right) \right) \mathbb{P}_n( M_m \cap H_m^c | \mathcal{F}_{2(m-1)}) \right] , \\
\mathbb{P}_{n,i}(\tau = 2m, M_m) &= \mathbb{E}_{n,i}\left[  1\left( M_{m-1} \cap J_0 \cap \bigcap_{k=1}^{m-1} \left( H_{k} \cap J_k  \right) \cap H_m \right) \mathbb{P}_n( M_m \cap J_m^c | \mathcal{F}_{2m-1}) \right].
\end{align*}
To analyze the two conditional probabilities inside the expectations above note that conditionally on $\mathcal{F}_{2(m-1)}$, the types of the nodes in $I_{m-1}$ are known and so are the nodes in $A_{m-1}$.  Therefore, by the union bound and the independence among the edges, we have:
\begin{align*}
\mathbb{P}_n( M_m \cap H_m^c | \mathcal{F}_{2(m-1)}) &= \mathbb{P}_n\left( \left. M_m \cap \bigcup_{i \in A_{m-1}} F_i(I_{m-1} \cup \mathbb{A}_i)^c \right| \mathcal{F}_{2(m-1)} \right)  \\
&\leq \sum_{i \in A_{m-1} }  \mathbb{P}_n\left( \left. M_m \cap F_i(I_{m-1} \cup \mathbb{A}_i )^c \right| \mathcal{F}_{2(m-1)} \right)  \\
&\leq \sum_{i \in A_{m-1} }   \min\left\{  1, \, \mathbb{P}_n\left( \left.  \max_{1 \leq j \leq n, j \neq i} |X_{ji} - Z_{ji}| \geq 1   \right| \mathcal{F}_{2(m-1)} \right) \right. \\
&\hspace{5mm} + \left.  \mathbb{P}_n\left( \left. M_m \cap \left\{ Z_{ii} +  \sum_{1 \leq j \leq n, j \neq i, j \in I_{m-1} \cup \mathbb{A}_i}  Z_{ji} \geq 1\right\}   \right| \mathcal{F}_{2(m-1)} \right)  \right\}   .
\end{align*}
Now use the independence of the edges from the rest of the exploration process and Lemma~\ref{L.EdgeDiscrepancy} to obtain that
\begin{align*}
 \mathbb{P}_n\left(  \left. \max_{1 \leq j \leq n, j \neq i} |X_{ji} - Z_{ji}| \geq 1 \right| \mathcal{F}_{2(m-1)} \right) &=  \mathbb{P}_n\left(  \max_{1 \leq j \leq n, j \neq i} |X_{ji} - Z_{ji}| \geq 1  \right) \\
 &\leq 1(W_i^- > b_n) + \mathcal{P}_n^-(i) + \bar W_i^- \eta_n^- .
\end{align*}
Next, condition further on the exploration up to the moment we are about to explore the inbound neighbors of $i$, and use the independence of the edges from the rest of the exploration process to obtain that
\begin{align*}
& \mathbb{P}_n\left(  \left. M_m \cap \left\{  Z_{ii} + \sum_{1 \leq j \leq n, j \neq i,  j \in I_{m-1} \cup \mathbb{A}_i} Z_{ji} \geq 1 \right\}  \right| \mathcal{F}_{2(m-1)} \right) \\
&\leq \mathbb{E}_n\left[ 1( |\mathbb{A}_i| \leq s_n) \left(  1 - e^{-q_{ii}^{(n)} -  \sum_{1 \leq j \leq n, j \neq i, j \in I_{m-1} \cup \mathbb{A}_i} q_{ji}^{(n)} }  \right) \right] \\
 &= \mathbb{E}_n\left[ 1(|\mathbb{A}_i| \leq s_n) \left( 1 - e^{-\frac{{\bar W}_i^-}{\theta n} \left( {\bar W}_i^+ +  \sum_{1 \leq j \leq n, j \neq i, j \in I_{m-1} \cup \mathbb{A}_i} {\bar W}_j^+  \right) } \right) \right]  \\
 &\leq \mathbb{E}_n\left[ 1(|\mathbb{A}_i| \leq s_n) \left(  1 - e^{-\frac{a_n {\bar W}_i^-}{\theta n} (1 + |I_{m-1}| + |\mathbb{A}_i| )}  \right) \right] \\
 &\leq  \frac{a_n {\bar W}_i^-}{\theta n} (1 + |I_{m-1}| + s_n) ,
\end{align*}
where in the last inequality we used $1 - e^{-x} \leq x$ for $x \geq 0$ and $|\mathbb{A}_i| \leq s_n$. 

It follows that 
\begin{align*}
\mathbb{P}_{n,i}(\tau = 2m-1, M_{m})  &\leq  \mathbb{E}_{n,i}\left[ 1\left( M_{m-1} \cap J_0 \cap  \bigcap_{k=1}^{m-1} \left( H_{k} \cap J_k  \right) \right)  \sum_{j \in A_{m-1}}  \min\left\{1, \,  \mathcal{P}_n^-(j) \phantom{\frac{\bar W^-}{\theta}} \right. \right.  \\
&\hspace{5mm} \left.  \left. +\eta_n^- \bar W_j^- + \frac{a_n {\bar W}_j^-}{\theta n} (1 + |I_{m-1}| + s_n)  \right\}  \right] .
\end{align*}
Almost the exact arguments, along with the observation that the event $|\hat A_m|$ is measurable with respect to $\mathcal{F}_{2m-1}$, can be used to obtain
\begin{align*}
\mathbb{P}_{n,i}(\tau = 2m, M_{m}) &\leq \mathbb{E}_{n,i}\left[  1\left( M_{m-1} \cap \{ |\hat A_m| \leq s_n\} \cap J_0 \cap \bigcap_{k=1}^{m-1} \left( H_{k} \cap J_k  \right) \cap H_m \right)  \right. \\
&\hspace{5mm} \left. \times  \sum_{j \in A_m} 1(\bar W_i^+ \leq c_n) \min\left\{ 1, \, \mathcal{P}_n^+(j)  + \eta_n^+ \bar W_j^+ + \frac{b_n {\bar W}_j^+}{\theta n} (1 + |T_{m}|)  \right\}   \right].
\end{align*}
To analyze these two remaining expectations we note that on the events $\{J_0 \cap  \bigcap_{k=1}^{m-1} \left( H_{k} \cap J_k \right) \}$ and $\{ J_0 \cap \bigcap_{k=1}^{m-1} \left( H_{k} \cap J_k  \right) \cap H_m\}$ the coupling has not broken yet, and therefore we can can replace $A_{m-1}$, $I_{m-1}$, $A_m$ and $T_{m}$ with their tree counterparts $\hat A_{m-1}$, $\hat I_{m-1}$, $\hat A_m$ and $\hat T_{m}$.  Also, note that by Lemma~\ref{L.Independence} we have that the types of the nodes in each of the active sets $\hat A_k$ are independent of the type of their parents. We will then identify the nodes in $\hat A_{m-1}$ (or $\hat A_{m-1}$) as $\{Y_1, \dots, Y_{|\hat A_{m-1}|}\}$ (or $\{ Y_1, \dots, Y_{|\hat A_m|}\}$), where for any $t \geq 1$, 
$$\mathbb{P}_n(Y_{t} = j) = \frac{{\bar W}_j^+}{\Lambda_n^+}, \quad j = 1, 2, \dots, n.$$

It follows that
\begin{align*}
\mathbb{P}_{n,i}(\tau = 2m-1, M_{m})  &\leq  \mathbb{E}_{n,i}\left[ 1\left( M_{m-1} \right) \sum_{t=1}^{|\hat A_{m-1}|}  \min\left\{ 1,  \, \mathcal{P}_n^-(Y_t ) + \eta_n^- \bar W_{Y_t}^-  \phantom{\frac{\bar W^-}{\theta}}  \right. \right. \\
&\hspace{5mm}  \left. \left. + \frac{a_n {\bar W}_{Y_t}^-}{\theta n} (1 + |\hat I_{m-1}| + s_n)    \right\}   \right] 
\end{align*}
and
\begin{align*}
\mathbb{P}_{n,i}(\tau = 2m, M_{m}) &\leq \mathbb{E}_{n,i}\left[  1\left( M_{m-1} \cap \{ |\hat A_m| \leq s_n\} \right)  \sum_{t=1}^{|\hat A_m|} \min \left\{ 1, \, \mathcal{P}_n^+(Y_{t}) + \eta_n^+ W_{Y_t}^+ \phantom{\frac{\bar W^-}{\theta}}  \right. \right.  \\ 
&\left. \left.  +  \frac{b_n {\bar W}_{Y_t}^+}{\theta n} (1 + |\hat T_{m}|)    \right\}  \right].
\end{align*}

Since on the event $M_{k}$ we have $|\hat I_{k}| \leq s_n \leq b_n$ and $|\hat A_k| \leq |\hat T_{k+1}| \leq s_n \leq a_n$, we further obtain that
\begin{align*}
\mathbb{P}_{n,i}(\tau = 2m-1, M_{m}) &\leq  \mathbb{E}_{n,i}\left[ \sum_{t=1}^{\lfloor s_n \rfloor}  \min \left\{ 1, \, \mathcal{P}_n^-(Y_t ) + \eta_n^- \bar W_{Y_t}^- + \frac{a_n {\bar W}_{Y_t}^-}{\theta n} (1 + 2b_n)    \right\}  \right] \\
&\leq s_n  \mathbb{E}_{n}\left[   \min\left\{ 1, \, \mathcal{P}_n^-(Y_1 ) + (\eta_n^- + a_n (1+2b_n)/(\theta n))  \bar W_{Y_1}^- \right\}   \right]
\end{align*}
and
\begin{align*}
\mathbb{P}_{n,i}(\tau = 2m, M_{m}) &\leq \mathbb{E}_{n,i}\left[  \sum_{t=1}^{\lfloor s_n \rfloor}  \min\left\{ 1, \, \mathcal{P}_n^+(Y_{t}) + \eta_n^+ \bar W_{Y_t}^+ +  \frac{b_n {\bar W}_{Y_t}^+}{\theta n} (1 +a_n)    \right\}   \right] \\
&\leq s_n \mathbb{E}_{n}\left[  \min \left\{ 1, \,  \mathcal{P}_n^+(Y_1) + (\eta_n^+ +b_n(1+a_n)/(\theta n)) \bar W_{Y_1}^+ \right\}  \right] .
\end{align*}

It only remains to compute the last two expectations. Throughout the rest of the proof, let $(W^{(-,n)}, W^{(+,n)}, W^-, W^+)$ be constructed according to the optimal coupling for $F_n$ and $F$, i.e., $\mathbb{E}_n\left[ |W^{(-,n)} - W^- | + |W^{(+,n)} - W^+ | \right] = \Delta_n$. Let $(\bar W^{(-,n)}, \bar W^{(+,n)}) = (W^{(-,n)} \wedge b_n, W^{(+,n)} \wedge a_n)$. Now let $\gamma_n^- = (E[W^-] + \Delta_n) (\eta_n^- + a_n(1+2b_n)/(\theta n))$ and note that for any $c_n \geq 1$, 
\begin{align*}
&s_n  \mathbb{E}_{n}\left[  \min\left\{ 1,\, \mathcal{P}_n^-(Y_1 ) + \frac{\gamma_n^-}{E[W^-] + \Delta_n} \bar W_{Y_1}^- \right\}     \right] \\
&= s_n \sum_{j=1}^n \frac{\bar W_j^+}{\Lambda_n^+} \min \left\{ 1, \, \mathcal{P}_n^-(j ) + \frac{\gamma_n^-}{E[W^-] + \Delta_n}  \bar W_j^- \right\}  \\
&\leq \frac{s_n n}{\Lambda_n^+} \cdot \frac{1}{n} \sum_{j=1}^n (\bar W_j^+ - c_n)^+ + \frac{s_n  n}{\Lambda_n^+} \cdot \frac{1}{n} \sum_{j=1}^n c_n \min \left\{ 1, \, \mathcal{P}_n^-(j ) + \frac{\gamma_n^-}{E[W^-]+\Delta_n} \bar W_j^- \right\}  \\
&\leq \frac{s_n}{\Lambda_n^+/n} \cdot \mathbb{E}_n[ (\bar W^{(+,n)} - c_n)^+] + \frac{s_n c_n}{\Lambda_n^+/n}  \left( \mathcal{E}_n + \frac{\gamma_n^-}{E[W^-] + \Delta_n} \mathbb{E}_n[ \bar W^{(-,n)}] \right) \\
&\leq \frac{s_n}{\Lambda_n^+/n} \left( g_+(c_n) + \Delta_n + c_n \mathcal{E}_n + c_n \gamma_n^-   \right). 
\end{align*}
Essentially the same arguments also yield for $\gamma_n^+ = (E[W^+] + \Delta_n) (\eta_n^+ + b_n(1+a_n)/(\theta n))$,
\begin{align*}
&s_n \mathbb{E}_{n}\left[  \min \left\{ 1, \,  \mathcal{P}_n^+(Y_1) + \frac{\gamma_n^+}{E[W^+] + \Delta_n} \bar W_{Y_1}^+ \right\}  \right]  \\
&\leq \frac{s_n}{\Lambda_n^+/n } \left( g_+(c_n) + \Delta_n + c_n \mathcal{E}_n + c_n \gamma_n^+  \right).
\end{align*}
This completes the proof. 
\end{proof}

We are now ready to prove Theorem~\ref{T.CouplingBreaks}.

\begin{proof}[Proof of Theorem~\ref{T.CouplingBreaks}]
Fix $c_n, s_n \geq 1$ with $s_n \leq a_n \wedge b_n$, and the event $M_m$ as in Proposition~\ref{P.CouplingBreaksAtStepM}. Now write
\begin{align*}
\mathbb{P}_{n,i}(\tau \leq 2k) &\leq \mathbb{P}_{n,i}(\tau \leq 2k, M_k)  + \mathbb{P}_{n,i}(M_k^c) \\
&= \mathbb{P}_{n,i}(\tau = 0, M_k) + \sum_{m=1}^{k} \left\{ \mathbb{P}_{n,i}(\tau = 2m-1, M_k) + \mathbb{P}_{n,i}(\tau = 2m, M_k) \right\} + \mathbb{P}_{n,i}(M_k^c) \\
&\leq \mathbb{P}_{n,i}(\tau = 0) + \sum_{m=1}^{k} \left\{ \mathbb{P}_{n,i}(\tau = 2m-1, M_{m}) + \mathbb{P}_{n,i}(\tau = 2m, M_m) \right\} + \mathbb{P}_{n,i}(M_k^c)  ,
\end{align*}
where in the last inequality we used the observation that $M_{m+1} \subseteq M_m$ for all $m \geq 1$. Now use Lemma~\ref{L.TauZero} and Proposition~\ref{P.CouplingBreaksAtStepM} to obtain that
\begin{align*}
\mathbb{P}_{n,i}(\tau \leq 2k) &\leq 1(W_i^+ > a_n)  +  \mathcal{P}_n^+(i) + \bar W_i^+ (\eta_n^+ + b_n/(\theta n)) \\
&\hspace{5mm} + \sum_{m=1}^k \frac{s_n}{\Lambda_n^+/n } \left( 2g_+(c_n) + 2\Delta_n + 2c_n \mathcal{E}_n + c_n (\gamma_n^-+ \gamma_n^+)   \right)  + \mathbb{P}_{n,i}(M_k^c) \\
&\leq  1(W_i^+ > a_n) + \mathcal{P}_n^+(i) +  (\mathcal{H}_n/\theta) \bar W_i^+ (\Delta_n + g_-(b_n) + a_n b_n/n)  \\
&\hspace{5mm} +  \mathcal{H}_n  k s_n \left(  g_+(c_n) + c_n \left( \mathcal{E}_n + \Delta_n + g_+(a_n) + g_-(b_n) + a_n b_n/n \right)  \right)  \\
&\hspace{5mm} + \mathbb{P}_{n,i}( |\hat T_{k+1}| \vee |\hat I_k| > s_n),
\end{align*}
where 
$$\mathcal{H}_n = \frac{2}{\Lambda_n^+/n}  (1+ \Delta_n/\theta)^2(3 + a_n^{-1} + b_n^{-1} + c_n^{-1} ).$$ 
\end{proof}

\subsection{Proof of Theorem~\ref{T.CouplingFinal}} \label{SS.CouplingCorollary}

In view of Theorem~\ref{T.CouplingBreaks}, the proof of Theorem~\ref{T.CouplingFinal} reduces to showing that we can choose $a_n, b_n, c_n, s_n$ such that the bound in Theorem~\ref{T.CouplingBreaks} converges to zero. The only term that is not yet explicit is 
$$\frac{1}{n} \sum_{i=1}^n \mathbb{P}_{n,i}\left(  |\hat T_{k+1}|  \vee |\hat I_k| > s_n \right) ,$$
which we will first write in terms of a marked Galton-Watson process that does not depend on the type sequence $\mathscr{W}_n = \{ {\bf W}_i: 1 \leq i \leq n\}$. To do this we need two preliminary results, the first of which shows  the convergence of the degree vectors $(\hat N_\emptyset, \hat D_\emptyset)$ and  $(\hat N_1, \hat D_1)$ in the total variation distance.

\begin{lemma} \label{L.TotalVariationDegrees}
Define $(\bar W^-, \bar W^+) = (W^- \wedge b_n, W^+ \wedge a_n)$ and consider the joint distributions
\begin{align*}
\mathbb{P}_{n}\left( \hat N_\emptyset = m, \, \hat D_\emptyset = k \right) &= \sum_{s=1}^n p(m; \Lambda_n^+  {\bar W}_s^-/(\theta n)) \cdot p(k; \Lambda_n^- {\bar W}_s^+/(\theta n)) \cdot \frac{1}{n}, \\
\mathbb{P}_{n}\left( \hat N_1 = m,  \hat D_1 = k \right) &= \sum_{s=1}^n p(m; \Lambda_n^+ {\bar W}_s^-/(\theta n)) \cdot p(k; \Lambda_n^- {\bar W}_s^+/(\theta n))\cdot \frac{{\bar W}_s^+ }{\Lambda_n^+}, \\
P\left( \bar{\mathcal{N}}_0 = m, \, \bar{\mathcal{D}}_0 = k \right) &= E\left[ p(m; E[ \bar W^+] \bar W^-/\theta) \cdot p(k; E[ \bar W^-] \bar W^+/\theta) \right], \\
P \left( \bar{\mathcal{N}} = m,  \bar{\mathcal{D}} = k \right) &= \frac{1}{E[W^+]} E\left[ W^+ p(m; E[ \bar W^+] \bar W^-/\theta) \cdot p(k; E[ \bar W^-] \bar W^+/\theta) \right],
\end{align*}
for $m,k \in \mathbb{N}$, where $p(m;\lambda) = e^{-\lambda} \lambda^m/m!$. Then, under Assumption~\ref{A.Types} (a)-(b), and for any $c_n \geq 1$,  
\begin{align*}
\sup_{A \subseteq \mathbb{N}^2} \left| \mathbb{P}_n\left( ( \hat N_\emptyset, \hat D_\emptyset) \in A \right) - P( (\bar{\mathcal{N}}_0, \bar{\mathcal{D}}_0) \in A) \right| &\leq  \frac{\Delta_n^2}{\theta} +2 \Delta_n, \\
\sup_{A \subseteq \mathbb{N}^2} \left| \mathbb{P}_n\left( ( \hat N_1, \hat D_1) \in A \right) - P( (\bar{\mathcal{N}}, \bar{\mathcal{D}}) \in A) \right| &\leq  \frac{1}{E[\bar W^+]} \left( 4c_n \Delta_n + 2g_+(a_n)  +g_+(c_n)   + c_n \Delta_n^2/\theta  \right) ,
\end{align*}
where $g_+$ is defined as in Lemma~\ref{L.EdgeDiscrepancy}.
\end{lemma}

\begin{proof}
By Assumption~\ref{A.Types} (a)-(b) and the observations following Definition~\ref{D.CouplingBreaks}, we know that $F_n$ converges to $F$ in the Kantorovich-Rubinstein distance, and therefore we can pick a random vector
$$\left(W^{(-,n)}, W^{(+,n)}, W^-, W^+ \right)$$
 so that $(W^{(-,n)}, W^{(+,n)})$ has distribution $F_n$, $(W^-, W^+)$ has distribution $F$, and
$$\mathbb{E}_n\left[ |W^{(-,n)} - W^-| + |W^{(+,n)} - W^+|  \right] = d_1(F_n, F) = \Delta_n.$$
Next, using this optimal coupling define $(\bar W^{(-,n)}, \bar W^{(+,n)}) = (W^{(-,n)} \wedge b_n, W^{(+,n)} \wedge a_n)$,  $X^{(-,n)} = \Lambda_n^+ \bar W^{(-,n)} /(\theta n)$, $X^{(+,n)} = \Lambda_n^- \bar W^{(+,n)} /(\theta n)$, $X^- = E[ \bar W^+] \bar W^-/\theta$, and $X^+ = E[\bar W^-] \bar W^+/\theta$. Let $\mathscr{G}_n = \sigma\left( \mathscr{F}_n \cup \sigma( W^{(-,n)}, W^{(+,n)}, W^-, W^+) \right)$. 

Now note that for any $A \subseteq \mathbb{N}^2$, 
\begin{align*}
& \left| \mathbb{P}_n\left( (\hat N_\emptyset, \hat D_\emptyset) \in A\right) - P\left((\mathcal{N}_0, \mathcal{D}_0) \in A  \right) \right| \\
&\leq   \mathbb{E}_n\left[ \left| \mathbb{P}_n\left( \left. (\hat N_\emptyset, \hat D_\emptyset) \in A \right| \mathscr{G}_n \right) - \mathbb{P}_n\left( \left. (\bar{\mathcal{N}}_0, \bar{\mathcal{D}}_0) \in A \right| \mathscr{G}_n  \right) \right| \right] \\
&\leq \mathbb{E}_n\left[ \sup_{B \subseteq \mathbb{N}} \left| \mathbb{P}_n( \hat N_\emptyset \in B | \mathscr{G}_n) -  \mathbb{P}_n( \bar{\mathcal{N}}_0 \in B | \mathscr{G}_n) \right| +  \sup_{B \subseteq \mathbb{N}} \left| \mathbb{P}_n( \hat D_\emptyset \in B | \mathscr{G}_n) -  \mathbb{P}_n( \bar{\mathcal{D}}_0 \in B | \mathscr{G}_n) \right|   \right]  \\
&\leq \mathbb{E}_n\left[ \min\{1,  | X^{(-,n)} - X^-|\} + \min\{ 1, | X^{(+,n)} - X^+| \}    \right] ,
\end{align*}
where in the first inequality we used the conditional independence of $(\hat N_\emptyset, \bar{\mathcal{N}}_0)$ and $(\hat D_\emptyset, \bar{\mathcal{D}}_0)$ given $\mathscr{G}_n$, and in the second one we used the observation that if $\text{Poi}(\lambda)$ denotes a Poisson random variable with mean $\lambda$, then
$$\sup_{A \in \mathbb{N}} | P(\text{Poi}(\mu) \in A) - P(\text{Poi}(\lambda) \in A) |  \leq P(\text{Poi}(|\mu - \lambda|) \geq 1) \leq \min\{ 1, |\mu-\lambda|\}.$$
Moreover, since $\Lambda_n^\pm/n = \mathbb{E}_n[ \bar W^{(\pm,n)}]$ and $\mathbb{E}_n[ |\bar W^{(\pm,n)} - \bar W^\pm| ] \leq  \mathbb{E}_n[ |W^{(\pm,n)} - W^\pm| ] $, we have that
\begin{align*}
\mathbb{E}_n\left[ |X^{(-,n)} - X^-| \right] &\leq \frac{\Lambda_n^+}{\theta n} \mathbb{E}_n\left[ | \bar W^{(-,n)}  - \bar W^-| \right] + E[\bar W^-] \left| \frac{\Lambda_n^+}{\theta n} - \frac{E[\bar W^+]}{\theta} \right| \\
&\leq \left( \frac{\mathbb{E}_n[ \bar W^{(+,n)} - \bar W^+] }{\theta} + \frac{E[\bar W^+]}{\theta} \right) \mathbb{E}_n\left[ |\bar W^{(-,n)}  - \bar W^- | \right] \\
&\hspace{5mm} + \frac{E[\bar W^-]}{\theta} \left| \mathbb{E}_n[ \bar W^{(+,n)} ]  - E[\bar W^+]  \right| \\
&\leq \left( \frac{\mathbb{E}_n[ |W^{(+,n)} - W^+ |] }{\theta} + \frac{E[W^+]}{\theta} \right) \mathbb{E}_n[ |W^{(-,n)}  - W^- | ] \\
&\hspace{5mm} + \frac{E[W^-]}{\theta} \mathbb{E}_n[ | W^{(+,n)} -W^+| ]  ,
\end{align*}
and similarly,
\begin{align*}
\mathbb{E}_n\left[ |X^{(+,n)} - X^+| \right] &\leq \left( \frac{\mathbb{E}_n[ | W^{(-,n)} - W^-| ] }{\theta} + \frac{E[W^-]}{\theta} \right) \mathbb{E}_n [ | W^{(+,n)}  - W^+| ] \\
&\hspace{5mm} + \frac{E[W^-]}{\theta}  \mathbb{E}_n[ |W^{(+,n)} - W^+| ] .
\end{align*}
Combining the two bounds we obtain
$$\mathbb{E}_n\left[ |X^{(-,n)} - X^-| + |X^{(+,n)} - X^+| \right]  \leq \frac{\Delta_n^2}{\theta} +2 \Delta_n.$$
Taking the supremum over all $A$ gives the first result. 

For the second result we start by noting that for any $A \subseteq \mathbb{N}^2$,
\begin{align*}
\mathbb{P}_n\left( (\hat N_1, \hat D_1) \in A\right) &= \mathbb{E}_n\left[ \frac{\bar W^{(+,n)}}{\mathbb{E}_n[ \bar W^{(+,n)}]} \mathbb{P}_n( (\hat N_\emptyset, \hat D_\emptyset) \in A | \mathscr{G}_n ) \right] 
\end{align*}
and
$$P\left((\bar{\mathcal{N}}, \bar{\mathcal{D}}) \in A  \right) = E \left[ \frac{W^+}{E[W^+]} P( (\bar{\mathcal{N}}_0, \bar{\mathcal{D}}_0) \in A | W^-, W^+) \right].$$
Hence, for any $A \subseteq \mathbb{N}^2$ and any $c_n \geq 1$, 
\begin{align*}
& \left| \mathbb{P}_n\left( (\hat N_1, \hat D_1) \in A\right) - P\left((\bar{\mathcal{N}}, \bar{\mathcal{D}}) \in A  \right) \right| \\
&\leq  \left| \mathbb{E}_n\left[ \frac{\bar W^{(+,n)} }{ \mathbb{E}_n[ \bar W^{(+,n)} ] }  \mathbb{P}_n( (\hat N_\emptyset, \hat D_\emptyset) \in A | \mathscr{G}_n ) \right] - \mathbb{E}_n\left[ \frac{\bar W^+ }{E[ \bar W^+ ] }  \mathbb{P}_n( (\hat N_\emptyset, \hat D_\emptyset) \in A | \mathscr{G}_n ) \right]  \right|  \\
&\hspace{5mm} +  \left|  \mathbb{E}_n\left[ \frac{\bar W^+ }{E[ \bar W^+ ] }  \mathbb{P}_n( (\hat N_\emptyset, \hat D_\emptyset) \in A | \mathscr{G}_n ) \right]  - \mathbb{E}_n \left[ \frac{\bar W^+ }{E[ \bar W^+] }  \mathbb{P}_n \left( \left. (\bar{\mathcal{N}}_0, \bar{\mathcal{D}}_0) \in A \right| \mathscr{G}_n  \right) \right]  \right| \\
&\hspace{5mm} +  \left| E\left[ \frac{\bar W^+ }{E[ \bar W^+] } P\left( \left. (\bar{\mathcal{N}}_0, \bar{\mathcal{D}}_0) \in A \right| W^-, W^+ \right) \right]  - E\left[ \frac{W^+}{E[W^+]} P( (\bar{\mathcal{N}}_0, \bar{\mathcal{D}}_0) \in A | W^-, W^+) \right] \right| \\
&\leq \mathbb{E}_n\left[ \left| \frac{\bar W^{(+,n)}}{\mathbb{E}_n[ \bar W^{(+,n)}]} - \frac{\bar W^+}{E[\bar W^+]} \right| \right] + \mathbb{E}_n\left[ \frac{\bar W^+}{E[\bar W^+]} \left|  \mathbb{P}_n( (\hat N_\emptyset, \hat D_\emptyset) \in A | \mathscr{G}_n ) - \mathbb{P}_n\left( \left. (\bar{\mathcal{N}}_0, \bar{\mathcal{D}}_0) \in A \right| \mathscr{G}_n  \right) \right| \right] \\
&\hspace{5mm} + E\left[ \left| \frac{\bar W^+}{E[\bar W^+]} - \frac{W^+}{E[W^+]} \right| \right] \\
&\leq \frac{\mathbb{E}_n\left[  | E[\bar W^+] \bar W^{(+,n)} - \mathbb{E}_n[ \bar W^{(+,n)}] \bar W^+| \right]}{\mathbb{E}_n[ \bar W^{(+,n)}] E[\bar W^+]} + \frac{E[| \bar W^+ E[W^+] - W^+ E[\bar W^+]|]}{E[\bar W^+] E[W^+]}  \\
&\hspace{5mm} + \mathbb{E}_n\left[ \frac{\bar W^+}{E[\bar W^+]} \left( \min\{ 1, |X^{(-,n)} - X^-|\} + \min\{ 1, |X^{(+,n)} - X^+| \} \right) \right] \\
&\leq \frac{  2\mathbb{E}_n[| \bar W^{(+,n)} - \bar W^+|] }{E[\bar W^+]} + \frac{2 E[ (W^+ - a_n)^+] }{E[ W^+]}  \\
&\hspace{5mm} + \mathbb{E}_n\left[ \frac{\bar W^+}{E[\bar W^+]} \left( \min\{ 1, |X^{(-,n)} - X^-|\} + \min\{ 1, |X^{(+,n)} - X^+| \} \right) \right] \\
&\leq \frac{2 \Delta_n}{E[\bar W^+]} + \frac{2 g_+(a_n)}{E[W^+]} + \frac{c_n}{E[\bar W^+]} \mathbb{E}_n\left[  |X^{(-,n)} - X^-|\ +  |X^{(+,n)} - X^+|\right] + \frac{E[(\bar W^+- c_n)^+]}{E[\bar W^+]} \\
&\leq \frac{1}{E[\bar W^+]} \left( 2 \Delta_n + 2g_+(a_n)  +g_+(c_n)  + 2 c_n \Delta_n + c_n \Delta_n^2/\theta  \right) .\end{align*}
This completes the proof. 
\end{proof}

The second technical result prior to the proof of Theorem~\ref{T.CouplingFinal} states the convergence in total variation of the processes $|\hat T_{k+1}|$ and $|\hat I_k|$. Note that the delayed marked Galton-Watson process appearing in the lemma still depends on $n$ via the truncation of $W^+ \wedge a_n$ and $W^- \wedge b_n$, but does not depend on the type sequence $\mathscr{W}_n = \{ {\bf W}_i: 1 \leq i \leq n\}$. In particular, by monotonicity of the Poisson distribution in its parameter, we have that under Assumption~\ref{A.Types} (a)-(b), 
$$(\bar{\mathcal{N}}_0, \bar{\mathcal{D}}_0) \nearrow (\mathcal{N}_0, \mathcal{D}_0) \quad \text{a.s.} \qquad \text{and} \qquad (\bar{\mathcal{N}}, \bar{\mathcal{D}}) \nearrow (\mathcal{N}, \mathcal{D}) \quad \text{a.s} ,$$
as $n \to \infty$, for well-defined random vectors $(\mathcal{N}_0, \mathcal{D}_0)$ and $(\mathcal{N}, \mathcal{D})$. Moreover, under Assumption~\ref{A.Types} (a)-(b) we have that $E[\mathcal{N}_0 + \mathcal{D}_0] < \infty$, although it is possible to have $E[\mathcal{N} + \mathcal{D}] = \infty$. If the latter happens, the probability $P\left( |\bar{\mathcal{T}}_{k+1}| \vee |\bar{\mathcal{I}}_k| > s_n \right)$ will still converge to zero as $s_n \to \infty$ for any fixed $k \in \mathbb{N}_+$, however, it may do so very slowly.

\begin{lemma} \label{L.TotalVariation}
Under Assumption~\ref{A.Types} (a)-(b) we have that for any fixed $k \geq 1$ and any $c_n, s_n \geq 1$,
\begin{align*}
\frac{1}{n} \sum_{i=1}^n \mathbb{P}_{n,i}\left(  |\hat T_{k+1}|  \vee |\hat I_k| > s_n \right) &\leq P\left( |\bar{\mathcal{T}}_{k+1}| \vee |\bar{\mathcal{I}}_k| > s_n \right) +\frac{\Delta_n^2}{\theta} + 2\Delta_n \\
&\hspace{5mm}  + \frac{k s_n}{E[W^+ \wedge a_n]} \left( 4c_n \Delta_n + 2g_+(a_n)  +g_+(c_n)   + c_n \Delta_n^2/\theta  \right),
 \end{align*}
where $|\bar{\mathcal{T}}_{k+1}| =  \sum_{m=0}^{k} |\bar{\mathcal{A}}_m|$, $|\bar{\mathcal{I}}_k| = \sum_{m=0}^k \sum_{{\bf i} \in \bar{\mathcal{A}}_m} \bar{\mathcal{D}}_{\bf i}$,  and $\bar{\mathcal{A}}_m$ is the set of individuals in the $m$th generation of a delayed marked Galton-Watson process whose root has offspring/mark distributed as $(\bar{\mathcal{N}}_0, \bar{\mathcal{D}}_0)$ and all other nodes have offspring/mark distributed as $(\bar{\mathcal{N}}, \bar{\mathcal{D}})$, as defined in Lemma~\ref{L.TotalVariationDegrees}.
\end{lemma}

\begin{proof}
We start by noting that under the measure $\mathbb{P}_n(\cdot) = n^{-1} \sum_{i=1}^n \mathbb{P}_{n,i}(\cdot)$,  $|\hat T_{k+1}|$ and $|\hat I_k|$ denote the total population and the sum of all the marks, up to generation $k$, on a marked Galton-Watson process whose offspring/mark distribution is that of $(\hat N_\emptyset, \hat D_\emptyset)$ for the root node and $(\hat N_1, \hat D_1)$ for all other nodes, as defined in Lemma~\ref{L.TotalVariationDegrees}. Next, let $(\hat N_\emptyset, \hat D_\emptyset, \bar{\mathcal{N}}_0, \bar{\mathcal{D}}_0)$ and $(\hat N_1, \hat D_1, \bar{\mathcal{N}}, \bar{\mathcal{D}})$ be couplings satisfying 
$$\sup_{A \subseteq \mathbb{N}^2} \left| \mathbb{P}_n( (\hat N_\emptyset, \hat D_\emptyset) \in A) - P( (\bar{\mathcal{N}}_0, \bar{\mathcal{D}}_0) \in A) \right| = \mathbb{P}_n( (\hat N_\emptyset, \hat D_\emptyset) \neq (\bar{\mathcal{N}}_0, \bar{\mathcal{D}}_0) )$$
and
$$\sup_{A \subseteq \mathbb{N}^2} \left| \mathbb{P}_n( (\hat N_1, \hat D_1) \in A) - P( (\bar{\mathcal{N}}, \bar{\mathcal{D}}) \in A) \right| = \mathbb{P}_n( (\hat N_1, \hat D_1) \neq (\bar{\mathcal{N}}, \bar{\mathcal{D}}) ),$$
which are guaranteed to exist (see, e.g., Theorem~2.12 in \cite{Hollander_2012}). Construct the two marked Galton-Watson processes simultaneously using this optimal coupling of the degree/mark vectors and let $\sigma = \inf\{ m \geq 0: (\hat N_{\bf i}, \hat D_{\bf i}) \neq ( \bar{\mathcal{N}}_{\bf i}, \bar{\mathcal{D}}_{\bf i}) \text{ for some } {\bf i} \in \bar{\mathcal{A}}_m \}$.

Now note that
\begin{align*}
\frac{1}{n} \sum_{i=1}^n \mathbb{P}_{n,i}\left(  |\hat T_{k+1}|  \vee |\hat I_k| > s_n \right) &= \mathbb{P}_n\left( |\hat T_{k+1}| \vee |\hat I_k| > s_n \right) \\
&\leq P\left( |\bar{\mathcal{T}}_{k+1}| \vee |\bar{\mathcal{I}}_k| > s_n \right) + \mathbb{P}_n\left(  |\hat T_{k+1}| \vee |\hat I_k|  > s_n \geq  |\bar{\mathcal{T}}_{k+1}| \vee |\bar{\mathcal{I}}_k|  \right) \\
&\leq P\left( |\bar{\mathcal{T}}_{k+1}| \vee |\bar{\mathcal{I}}_k| > s_n \right) + \sum_{m=0}^k \mathbb{P}_n\left(    |\bar{\mathcal{T}}_{k+1}| \vee |\bar{\mathcal{I}}_k| \leq s_n , \, \sigma = m \right) . 
\end{align*}
To analyze each of the probabilities in the sum let $\mathcal{F}_m = \sigma\left( (\hat N_{\bf i}, \hat D_{\bf i}, \bar{\mathcal{N}}_{\bf i}, \bar{\mathcal{D}}_{\bf i}): {\bf i} \in \hat A_l, 0 \leq l \leq m\right)$ and note that for $0 \leq m \leq k$:
\begin{align*}
\mathbb{P}_n\left(    |\bar{\mathcal{T}}_{k+1}| \vee |\bar{\mathcal{I}}_k| \leq s_n , \, \sigma = m \right)  &\leq  \mathbb{P}_n\left(    |\bar{\mathcal{A}}_{m}|  \leq s_n , \, \sigma = m \right) \\
&= \mathbb{E}_n\left[ 1( |\bar{\mathcal{A}}_{m}|  \leq s_n ) \, \mathbb{P}_n(\sigma = m | \mathcal{F}_{m-1} ) \right] \\
&\leq \mathbb{E}_n\left[ 1( |\bar{\mathcal{A}}_{m}|  \leq s_n )  \, \mathbb{P}_n\left( \left.  \bigcup_{{\bf i} \in \bar{\mathcal{A}}_m} \{ (\hat N_{\bf i}, \hat D_{\bf i}) \neq (\bar{\mathcal{N}}_{\bf i}, \bar{\mathcal{D}}_{\bf i}) \}  \right| \mathcal{F}_{m-1} \right) \right] \\
&\leq \mathbb{E}_n\left[ 1( |\bar{\mathcal{A}}_{m}|  \leq s_n )   \sum_{{\bf i} \in \bar{\mathcal{A}}_m}   \mathbb{P}_n\left(   (\hat N_{\bf i}, \hat D_{\bf i}) \neq (\bar{\mathcal{N}}_{\bf i}, \bar{\mathcal{D}}_{\bf i})  \right) \right] \\
&= \mathbb{E}_n\left[ 1( |\bar{\mathcal{A}}_{m}|  \leq s_n )  ( |\bar{\mathcal{A}}_{m}| -1)  \mathbb{P}_n\left(   (\hat N_1, \hat D_1) \neq (\bar{\mathcal{N}}, \bar{\mathcal{D}})  \right) \right] 1(m \geq 1) \\
&\hspace{5mm} +  \mathbb{P}_n\left(   (\hat N_\emptyset, \hat D_\emptyset) \neq (\bar{\mathcal{N}}_0, \bar{\mathcal{D}}_0)  \right)  \\
&\leq s_n  \mathbb{P}_n\left(   (\hat N_1, \hat D_1) \neq (\bar{\mathcal{N}}, \bar{\mathcal{D}})  \right) 1(m \geq 1) \\
&\hspace{5mm} + \mathbb{P}_n\left(   (\hat N_\emptyset, \hat D_\emptyset) \neq (\bar{\mathcal{N}}_0, \bar{\mathcal{D}}_0)  \right) .
\end{align*}

It follows that
\begin{align*}
\frac{1}{n} \sum_{i=1}^n \mathbb{P}_{n,i}\left(  |\hat T_{k+1}|  \vee |\hat I_k| > s_n \right) &\leq P\left( |\bar{\mathcal{T}}_{k+1}| \vee |\bar{\mathcal{I}}_k| > s_n \right)  + \mathbb{P}_n\left(   (\hat N_\emptyset, \hat D_\emptyset) \neq (\bar{\mathcal{N}}_0, \bar{\mathcal{D}}_0)  \right) \\
&\hspace{5mm} + k s_n  \mathbb{P}_n\left(   (\hat N_1, \hat D_1) \neq (\bar{\mathcal{N}}, \bar{\mathcal{D}})  \right) .
\end{align*}
The conclusion now follows from Lemma~\ref{L.TotalVariationDegrees}.
\end{proof}

We now use these two results to prove Theorem~\ref{T.CouplingFinal}.

\begin{proof}[Proof of Theorem~\ref{T.CouplingFinal}]
Let $(W^{(-,n)}, W^{(+,n)})$ be distributed according to $F_n$. Then, by Theorem~\ref{T.CouplingBreaks} and Lemma~\ref{L.TotalVariation} we have that
\begin{align*}
\frac{1}{n} \sum_{i=1}^n \mathbb{P}_{n,i}(\tau \leq 2k) &\leq \mathbb{P}_n( W^{(-,n)} > a_n)+ \mathcal{E}_n + (\mathcal{H}_n/\theta) \frac{\Lambda_n^+}{n} (\Delta_n + g_-(b_n) + a_n b_n/n) \\
&\hspace{5mm} + \mathcal{H}_n k s_n \left( g_+(c_n) + c_n (\mathcal{E}_n + \Delta_n + g_+(a_n) + g_-(b_n) + a_n b_n/n ) \right)  \\
&\hspace{5mm} + P\left( |\bar{\mathcal{T}}_k| \vee |\bar{\mathcal{I}}_k| > s_n \right) +\frac{\Delta_n^2}{\theta} + 2\Delta_n \\
&\hspace{5mm}  + \frac{k s_n}{E[W^+ \wedge a_n]} \left( 4c_n \Delta_n + 2g_+(a_n)  +g_+(c_n)   + c_n \Delta_n^2/\theta  \right) \\
&\leq \mathbb{P}_n( W^{(-,n)} > a_n) + P\left( |\bar{\mathcal{T}}_{k+1}| \vee |\bar{\mathcal{I}}_k| > s_n \right) \\
&\hspace{5mm} +  \mathcal{K}_n  k s_n \left( g_+(c_n) + c_n (\mathcal{E}_n + \Delta_n + g_+(a_n) + g_-(b_n) + a_n b_n/n) \right),
\end{align*}
with 
$$\mathcal{K}_n = \frac{2}{\mathbb{E}_n[ W^{(-,n)} \wedge a_n] \wedge E[ W^+ \wedge a_n]}  (1+ \Delta_n/\theta)^2(3 + a_n^{-1} + b_n^{-1} + c_n^{-1} ) ,$$
and $\bar{\mathcal{T}}_{k+1}, \bar{\mathcal{I}}_k$ defined as in Lemma~\ref{L.TotalVariation}. As argued right before the statement of Lemma~\ref{L.TotalVariation}, $\lim_{n \to \infty} |\bar{\mathcal{T}}_{k+1}| \vee |\bar{\mathcal{I}}_k| < \infty$ a.s. for any fixed $k \in \mathbb{N}_+$, and therefore,
$$\lim_{n \to \infty} P\left( |\bar{\mathcal{T}}_{k+1}| \vee |\bar{\mathcal{I}}_k| > s_n \right) = 0$$
provided $s_n \to \infty$. Clearly, $\lim_{n \to \infty} \mathbb{P}_n( W^{(-,n)} > a_n) \leq \lim_{n \to \infty} a_n^{-1} \mathbb{E}_n[ W^{(-,n)}] =  0$ for any $a_n \to \infty$, so it only remains to show that we can pick $a_n, b_n, c_n, s_n$ such that $\min\{a_n, b_n, c_n, s_n\} \xrightarrow{P} \infty$  and
\begin{equation} \label{eq:ChooseSequences}
s_n \left( g_+(c_n) + c_n (\mathcal{E}_n + \Delta_n + g_+(a_n) + g_-(b_n) + a_n b_n/n) \right) \xrightarrow{P} 0
\end{equation}
as $n \to \infty$. To this end, choose $a_n = b_n = n^{(1-\epsilon)/2}$ for some $0 < \epsilon < 1$, 
$$c_n = (\mathcal{E}_n + \Delta_n + g_+(a_n) + g_-(b_n) + a_n b_n/n)^{-1/2},$$
and $s_n = ( g_+(c_n) + c_n^{-1})^{-1/2}$. Assumption~\ref{A.Types} (a)-(c) guarantee that $c_n \xrightarrow{P} \infty$ as $n \to \infty$ and our choice of $s_n$ ensures \eqref{eq:ChooseSequences} holds. 
\end{proof}

 \label{SS.DegreeProof}
\subsection{Proof of Theorem~\ref{T.MixedPoisson}}

We now give a short proof of Theorem~\ref{T.MixedPoisson} using Theorem~\ref{T.CouplingBreaks} and Lemma~\ref{L.TotalVariationDegrees}. A direct proof is possible, but would involve repeating some of the arguments used earlier.

\begin{proof}[Proof of Theorem~\ref{T.MixedPoisson}]
Start by noting that $(Z^-, Z^+) \stackrel{\mathcal{D}}{=} (\mathcal{N}_0, \mathcal{D}_0)$, as defined in Lemma~\ref{L.TotalVariationDegrees}, and therefore,,
\begin{align*}
&\sup_{A \subseteq \mathbb{N}^2} \left| \mathbb{P}_n( (D_\xi^-, D_\xi^+) \in A) - P( (Z^-, Z^+) \in A) \right| \\
&\leq \sup_{A \subseteq \mathbb{N}^2}\left| \mathbb{P}_n( (D_\xi^-, D_\xi^+) \in A) - \mathbb{P}_{n}((\hat N_\emptyset, \hat D_\emptyset) \in A) \right| + \sup_{A \subseteq \mathbb{N}^2} \left| \mathbb{P}_{n}((\hat N_\emptyset, \hat D_\emptyset) \in A) - P( (\mathcal{N}_0, \mathcal{D}_0) \in A) \right| \\
&\leq \frac{1}{n} \sum_{i=1}^n \mathbb{P}_{n,i}( \tau \leq 1) + \sup_{A \subseteq \mathbb{N}^2} \left| \mathbb{P}_{n}((\hat N_\emptyset, \hat D_\emptyset) \in A) - P( (\mathcal{N}_0, \mathcal{D}_0) \in A) \right|.
\end{align*}
Now use Lemma~\ref{L.TauZero} and Proposition~\ref{P.CouplingBreaksAtStepM} as in the proof of Theorem~\ref{T.CouplingBreaks} to obtain that
\begin{align*}
\mathbb{P}_{n,i}(\tau \leq 1) &\leq \mathbb{P}_{n,i} (\tau = 0) + \mathbb{P}_{n,i}(\tau = 1, M_0) + \mathbb{P}_{n,i}(M_0^c) \\
&\leq 1(W_i^+ > a_n) + \mathcal{P}_{n,i}(i) + (\eta_n^+ + b_n/(\theta n)) \bar W_i^+ \\
&\hspace{5mm} + \frac{s_n}{\Lambda_n^+/n} (g_+(c_n) + \Delta_n + c_n \mathcal{E}_n + c_n \gamma_n^-)  + \mathbb{P}_{n,i}(|\hat I_0| > s_n).
\end{align*}
Note that under $\mathbb{P}_n(\cdot) = n^{-1} \sum_{i=1}^n \mathbb{P}_{n,i}(\cdot)$ we have that $|\hat I_0| \stackrel{\mathcal{D}}{=} \hat D_\emptyset$, so following the same steps as in the proof of Theorem~\ref{T.CouplingFinal}, we obtain
\begin{align*}
\frac{1}{n} \sum_{i=1}^n \mathbb{P}_{n,i}(\tau \leq 1) &\leq \mathbb{P}_n( W^{(-,n)} > a_n) + \mathbb{P}_n( \hat D_\emptyset > s_n) \\
&\hspace{5mm} +  \mathcal{K}_n  s_n \left( g_+(c_n) + c_n (\mathcal{E}_n + \Delta_n + g_+(a_n) + g_-(b_n) + a_n b_n/n ) \right) ,
\end{align*}
where $(W^{(-,n)}, W^{(+,n)})$ is distributed according to $F_n$ and $\mathcal{K}_n$ is bounded. Moreover, this bound converges to zero as $n \to \infty$ for the same choice of $a_n, b_n, c_n, s_n$ used in the proof of Theorem~\ref{T.CouplingFinal}.  Finally, Lemma~\ref{L.TotalVariationDegrees} gives that
\begin{align*}
\sup_{A \subseteq \mathbb{N}^2} \left| \mathbb{P}_{n}((\hat N_\emptyset, \hat D_\emptyset) \in A) - P( (\mathcal{N}_0, \mathcal{D}_0) \in A) \right| \leq \frac{\Delta_n^2}{\theta} + 2\Delta_n \xrightarrow{P} 0
\end{align*}
as $n \to \infty$. We conclude that
$$\sup_{A \subseteq \mathbb{N}^2} \left| \mathbb{P}_n( (D_\xi^-, D_\xi^+) \in A) - P( (Z^-, Z^+) \in A) \right|  \xrightarrow{P} 0 \qquad n \to \infty.$$

To obtain the convergence of the means, let $(W^{(-,n)}, W^{(+,n)})$ and $(\hat W^{(-,n)}, \hat W^{(+,n)})$ be conditionally  i.i.d.~(given $\mathscr{F}_n$) vectors have distribution $F_n$ and note that
\begin{align*}
\left| \mathbb{E}_n[ D_\xi^\pm] - E[Z^\pm] \right| &\leq \left| \frac{1}{n} \sum_{i=1}^n \sum_{1\leq j \leq n, j \neq i} \left(p_{ij}^{(n)} - (r_{ij}^{(n)} \wedge 1) \right) \right| + \frac{1}{n} \sum_{i=1}^n \sum_{1\leq j \leq n, j \neq i}  r_{ij}^{(n)} 1( r_{ij}^{(n)} > 1)   \\
&\hspace{5mm} + \left| \frac{1}{n} \sum_{i=1}^n \sum_{1\leq j \leq n, j \neq i}  r_{ij}^{(n)}  - E[Z^\pm] \right| \\
&\leq \mathcal{E}_n + \frac{1}{\theta} \mathbb{E}_n\left[ W^{(+,n)} \hat W^{(-,n)} 1( W^{(+,n)} \hat W^{(-,n)} > \theta n) \right] \\
&\hspace{5mm} +\frac{1}{\theta} \left| \mathbb{E}_n[ W^{(-,n)}] \mathbb{E}_n[ W^{(+,n)}] - E[W^-] E[W^+] \right|,
\end{align*}
where $\mathcal{E}_n \xrightarrow{P} 0$ by Assumption~\ref{A.Types}(c) and $\mathbb{E}_n[ W^{(-,n)}] \mathbb{E}_n[ W^{(+,n)}]  \xrightarrow{P} E[W^+] E[W^-]$ by Assumption~\ref{A.Types}(b). For the middle term note that
\begin{align*}
& \mathbb{E}_n\left[ W^{(+,n)} \hat W^{(-,n)} 1( W^{(+,n)} \hat W^{(-,n)} > \theta n) \right] \\
&\leq \mathbb{E}_n\left[ W^{(+,n)} \hat W^{(-,n)} \left( 1( W^{(+,n)}  > \sqrt{\theta n}) + 1( \hat W^{(-,n)}  > \sqrt{\theta n}) \right) \right] \\
&= \mathbb{E}_n[W^{(-,n)}]  \mathbb{E}_n\left[ W^{(-,n)} 1( W^{(-,n)}  > \sqrt{\theta n})  \right] + \mathbb{E}_n[W^{(+,n)}]  \mathbb{E}_n\left[ W^{(+,n)} 1( W^{(+,n)}  > \sqrt{\theta n})  \right] ,
\end{align*}
which also converges in probability to zero as $n \to \infty$ since $E[ W^- + W^+] < \infty$. 

The result for the mixed expectation is a consequence of Assumption~\ref{A.Types} (d) since
$$\mathbb{E}_n\left[ D_\xi^- D_\xi^+ \right] = \frac{1}{n} \sum_{i=1}^n \sum_{1 \leq j \leq n, j \neq i} \sum_{1 \leq k \leq n, k \neq i} p_{ji}^{(n)} p_{ik}^{(n)} \xrightarrow{P} \frac{E[W^- W^+] E[W^-] E[W^+]}{\theta^2} = E[ Z^- Z^+],$$
as $n \to \infty$. This completes the proof. 
\end{proof}

\subsection{Proof of Theorem~\ref{T.SFPEConvergence}} \label{S.WBPconvergenceProofs}

In this section we prove Theorem~\ref{T.SFPEConvergence}, which establishes the convergence of $\hat R_\emptyset^{(n,k_n)}$ to the attracting endogenous solution to \eqref{eq:SFPE}, under Assumption~\ref{A.ExtendedTypes}. The main step in the proof of Theorem~\ref{T.SFPEConvergence} consists in showing that the vectors $\{ (\hat N_{\bf i}, \hat Q_{\bf i}, \hat C_{\bf i}): {\bf i} \in \mathcal{U} \}$ converge, in the Kantorovich-Rubinstein metric, to the distribution of $(\mathcal{N}, \mathcal{Q}, \mathcal{C})$ defined in Theorem~\ref{T.MainPageRank}, with $\mathcal{C}$ independent of $(\mathcal{N}, \mathcal{Q})$. To simplify the proof of Theorem~\ref{T.SFPEConvergence}, we show this convergence separately.

Throughout the section, for probability measures $\phi, \chi$ in $\mathbb{R}^d$, we interchangeably use the notation $d_1(\phi, \chi) = d_1(F, G)$ to denote the Kantorovich-Rubinstein distance between $\phi$ and $\chi$, where $F$ and $G$ are the cumulative distribution functions of $\phi$ and $\chi$, respectively.

\begin{theo} \label{T.KRconvergence}
Define $G_n^*(m,q) = \mathbb{P}_{n}( \hat N_\emptyset \leq m, \, \hat Q_\emptyset \leq q)$ and $G_n(m,q,t) = \mathbb{P}_{n}( \hat N_{\bf i} \leq m, \, \hat Q_{\bf i} \leq q, \, \hat C_{\bf i} \leq t)$ for ${\bf i} \neq \emptyset$, according to \eqref{eq:RootVector} and \eqref{eq:TypicalVector}, respectively. Define $G^*(m,q) = P(\mathcal{N}_0 \leq m, \mathcal{Q}_0 \leq q)$ and $G(m,q,t) = P(\mathcal{N} \leq m, \mathcal{Q} \leq q, \mathcal{C} \leq t)$, for $m \in \mathbb{N}$, $q,t \in \mathbb{R}$, according to:
\begin{align*}
P(\mathcal{N}_0 = m, \mathcal{Q}_0 \in dq) &= E\left[1\left(Q \in dq \right) p(m; E[W^+] W^-/\theta) \right],  \\
P(\mathcal{N} = m, \mathcal{Q} \in dq, \mathcal{C} \in dt) &= E\left[ \frac{W^+ }{E[W^+]}  1\left(Q \in dq, \zeta/(Z^++1) \in dt \right) p(m; E[W^+] W^-/\theta)  \right], 
\end{align*}
where $Z^+$ is a mixed Poisson random variable with mixing parameter $E[W^-] W^+/\theta$ and $p(m;\lambda) = e^{-\lambda} \lambda^m/m!$. Then, under Assumption~\ref{A.ExtendedTypes} (a)-(b), we have that
$$d_1(G_n^*,, G^*) \xrightarrow{P} 0 \qquad \text{and} \qquad (\hat N_1, \, \hat Q_1, \hat C_1)  \Rightarrow (\mathcal{N}, \mathcal{Q}, \mathcal{C})$$
as $n \to \infty$. Moreover, if Assumption~\ref{A.ExtendedTypes} (a), (b), (e) hold, then
$$d_1(G_n, G) \xrightarrow{P} 0 \quad \text{as} n \to \infty.$$
\end{theo}

\begin{proof}
We start by showing the weak convergence of the vectors $(\hat N_\emptyset, \hat Q_\emptyset)$ and $(\hat N_1, \hat D_1, \hat Q_1, \hat \zeta_1)$; recall that $\hat C_1 = \hat \zeta_1/(\hat D_1 + 1)$. Let $( W^{(-,n)}, W^{(+,n)}, Q^{(n)}, \zeta^{(n)})$ be distributed according to $H_n$ and define $(\bar W^{(-,n)}, \bar W^{(+.n)}) = (W^{(-,n)} \wedge b_n, W^{(+,n)} \wedge a_n)$. Now note that if $f: \mathbb{N} \times \mathbb{R} \to \mathbb{R}$ is bounded and continuous, then the function $J(x,q) = \sum_{m=0}^\infty f(m,q) p(m; x)$ is also bounded and continuous, and by Assumption~\ref{A.ExtendedTypes} (a)-(b) we have
\begin{align*}
\mathbb{E}_n\left[ f(\hat N_\emptyset, \hat Q_\emptyset) \right] &= \frac{1}{n} \sum_{s=1}^n \sum_{m=0}^\infty f(m, Q_s) p(m; \Lambda_n^+ {\bar W}_s^-/(\theta n)) \\
&= \mathbb{E}_n\left[ J\left( Q^{(n)}, \Lambda_n^+ \bar W^{(-,n)} /(\theta n) \right) \right] \\
&\xrightarrow{P} E\left[ J(Q, E[W^+] W^-/\theta) \right] = E[ f(\mathcal{N}_0, \mathcal{Q}_0)],
\end{align*}
as $n \to \infty$. For the second vector let $g: \mathbb{N} \times \mathbb{R}^2 \to \mathbb{R}$ be a bounded and continuous function and note that $J_M(x,y,q,z) = (y \wedge M) \sum_{m=0}^\infty \sum_{k=0}^\infty p(m;x) p(k;y) g(m,q,z/(k+1))$  is bounded and continuous for any fixed $M > 0$. Then,
\begin{align*}
\mathbb{E}_n\left[ g(\hat N_1, \hat Q_1, \hat C_1) \right] &= \sum_{s=1}^n \frac{{\bar W}_s^+}{\Lambda_n^+} \sum_{m=0}^\infty \sum_{k=0}^\infty p(m; \Lambda_n^+ {\bar W}^-_s/(\theta n)) p(k; \Lambda_n^- {\bar W}^+_s/(\theta n))  g( m, Q_s, \zeta_s/(k + 1) ) \\
&\leq \frac{\theta n^2}{\Lambda_n^+ \Lambda_n^-} \mathbb{E}_n\left[ J_M\left( \Lambda_n^+ \bar W^{(-,n)}/(\theta n), \Lambda_n^- \bar W^{(+,n)}/(\theta n), Q^{(n)}, \zeta^{(n)} \right) \right]  \\
&\hspace{5mm} + \frac{n}{\Lambda_n^+} \mathbb{E}_n\left[ (\bar W^{(+,n)}  - M)^+ \right] \sup_{m,q,z}|g(m,q,z)| ,
\end{align*}
and, similarly,
\begin{align*}
\mathbb{E}_n\left[ g(\hat N_1, \hat Q_1, \hat C_1) \right]  &\geq \frac{\theta n^2}{\Lambda_n^+ \Lambda_n^-} \mathbb{E}_n\left[ J_M\left( \Lambda_n^+ \bar W^{(-,n)} /(\theta n), \Lambda_n^- \bar W^{(+,n)} /(\theta n), Q^{(n)}, \zeta^{(n)} \right) \right]  \\
&\hspace{5mm} - \frac{n}{\Lambda_n^+} \mathbb{E}_n\left[ (\bar W^{(+,n)}  - M)^+ \right] \sup_{m,q,z}|g(m,q,z)|  .
\end{align*}
It follows from Assumption~\ref{A.ExtendedTypes} (a)-(b) that the limits
\begin{align*}
&\frac{\theta}{E[W^+] E[W^-]} E\left[ J_M(E[W^+] W^-/\theta, E[W^-] W^+/\theta, Q, \zeta) \right] - \frac{E[ (W^+ - M)^+]}{E[W^+]}  \sup_{m,q,z}|g(m,q,z)|   \\
&\leq \liminf_{n \to \infty} \mathbb{E}_n\left[ g(\hat N_1, \hat Q_1, \hat C_1) \right]  \leq \limsup_{n \to \infty} \mathbb{E}_n\left[ g(\hat N_1, \hat Q_1, \hat C_1) \right]  \\
&\leq \frac{\theta}{E[W^+] E[W^-]} E\left[ J_M(E[W^+] W^-/\theta, E[W^-] W^+/\theta, Q, \zeta) \right] + \frac{E[ (W^+ - M)^+]}{E[W^+]}  \sup_{m,q,z}|g(m,q,z)| 
\end{align*}
hold in probability. Taking the limit as $M \to \infty$ now yields (via the dominated convergence theorem)
\begin{align*}
\mathbb{E}_n\left[ g(\hat N_1, \hat Q_1, \hat C_1) \right]  &\xrightarrow{P} \frac{\theta}{E[W^+] E[W^-]} \lim_{M \to \infty} E\left[ J_M (E[W^+] W^-/\theta, E[W^-] W^+/\theta, Q, \zeta) \right] \\
&= E\left[ \frac{W^+}{E[W^+]} \sum_{m=0}^\infty \sum_{k=0}^\infty p(m; E[W^+] W^-/\theta) p(k; E[W^-] W^+/\theta) g(m, Q, \zeta/(k+1)) \right] \\
&= E\left[ g(\mathcal{N}, \mathcal{Q}, \mathcal{C}) \right]. 
\end{align*}
This establishes the weak convergence for both $G_n^*$ and $G_n$. 

To prove the convergence in the Kantorovich-Rubinstein distance recall that it suffices to show that the first absolute moments converge (see Theorem~6.9 and Definition~6.8(i) in \cite{Villani_2009}). Under Assumption~\ref{A.ExtendedTypes} (a)-(b) we have that
\begin{align*}
\mathbb{E}_n\left[ \hat N_\emptyset + |\hat Q_\emptyset| \right] &= \mathbb{E}_n\left[ \frac{\Lambda_n^+ \bar W^{(-,n)} }{\theta n} + |Q^{(n)} | \right]  \xrightarrow{P} E\left[ \frac{E[W^+] W^-}{\theta} + |Q| \right] = E\left[ \mathcal{N}_0 + |\mathcal{Q}_0| \right]
\end{align*}
as $n \to \infty$. We conclude that $d_1(G_n^*, G^*) \xrightarrow{P} 0$ as $n \to \infty$. 

For $G_n$ we have that
\begin{align*}
&\mathbb{E}_n\left[ \hat N_1 + |\hat Q_1| + |\hat C_1| \right] \\
&= \mathbb{E}_n\left[ \hat N_1 + |\hat Q_1| + \frac{|\hat \zeta_1|}{\hat D_1 + 1} \right] \\
&= \mathbb{E}_n\left[  \frac{n \bar W^{(+,n)} }{\Lambda_n^+}   \left( \frac{\Lambda_n^+ \bar W^{(-,n)} }{\theta n} + |Q^{(n)}|  + \frac{|\zeta^{(n)}|}{\Lambda_n^- \bar W^{(+,n)} /(\theta n)} \left(1 - e^{-\Lambda_n^- \bar W^{(+,n)} /(\theta n)} \right) \right) \right] \\
&= \mathbb{E}_n\left[  \frac{ \bar W^{(+,n)}  \bar W^{(1,n)} }{\theta } + \frac{n \bar W^{(+,n)} |Q^{(n)}|}{\Lambda_n^+}  + \frac{\theta n^2 |\zeta^{(n)}|}{\Lambda_n^- \Lambda_n^+} \left(1 - e^{-\Lambda_n^- \bar W^{(+,n)}/(\theta n)} \right) \right] ,
\end{align*}
where we used the observation that if $Z$ is Poisson with mean $\lambda$ then $E[1/(Z+1)] = \frac{1}{\lambda} (1 - e^{-\lambda})$. The third summand inside the expectation converges under Assumption~\ref{A.ExtendedTypes} (a)-(b), however, the first two require part (e) of the assumption. Hence, under Assumption~\ref{A.ExtendedTypes} (a),(b),(e) we have
\begin{align*}
\mathbb{E}_n\left[ \hat N_1 + |\hat Q_1| + |\hat C_1| \right] &\xrightarrow{P} E\left[ \frac{W^+ W^-}{\theta} + \frac{W^+ |Q|}{E[W^+]} + \frac{\theta |\zeta|}{E[W^-] E[W^+]} \left( 1 - e^{-E[W^-] W^+/\theta} \right) \right] \\
&= E\left[ \frac{W^+}{E[W^+]} \left( \frac{E[W^+] W^-}{\theta} + |Q| + \frac{|\zeta|}{E[W^-] W^+/\theta} \left( 1 - e^{-E[W^-] W^+/\theta} \right) \right) \right] \\
&= E\left[ \mathcal{N} + |\mathcal{Q}| + |\mathcal{C}| \right]
\end{align*}
as $n \to \infty$. This completes the proof. 
\end{proof}

The proof of Theorem~\ref{T.SFPEConvergence} will now follow from Theorem~2 in  \cite{chenolvera1}.

\begin{proof}[Proof of Theorem~\ref{T.SFPEConvergence}]
Recall that the sequence $\{ (\hat N_{\bf i}, \hat Q_{\bf i}, \hat C_{\bf i}): {\bf i} \in \mathcal{U}, {\bf i} \neq \emptyset \}$ consists of conditionally i.i.d.~vectors, given $\mathscr{F}_n$, with $(\hat N_\emptyset, \hat Q_\emptyset, \hat C_\emptyset)$ conditionally independent of this sequence. To simplify the notation let $(\hat Q, \hat N, \hat C) \stackrel{\mathcal{D}}{=} (\hat Q_1, \hat N_1, \hat C_1)$. Define $\phi_n$ to be the probability measure of the vector 
$$(\hat C \hat Q, \hat C 1(\hat N \geq 1), \hat C 1(\hat N \geq 2), \dots)$$
and let $\phi_n^*$ denote the probability measure of $(\hat Q_\emptyset, \hat N_\emptyset)$. Similarly, define $\phi$ and $\phi^*$ to be the probability measures of vectors
$$(\mathcal{C} \mathcal{Q}, \mathcal{C} 1(\mathcal{N} \geq 1), \mathcal{C} 1(\mathcal{N} \geq 2), \dots) \qquad \text{and} \qquad (\mathcal{Q}, \mathcal{N}),$$
respectively, where $\mathcal{C}$ and $(\mathcal{N}, \mathcal{Q})$ are distributed as in Theorem~\ref{T.KRconvergence}, with $\mathcal{C}$ independent of $(\mathcal{Q}, \mathcal{N})$.

Next, let $d_1$ denote the Kantorovich-Rubinstein distance on $\mathcal{S}$, with $\mathcal{S}$ either $\mathbb{R}^\infty$ or $\mathbb{R}^2$ as needed, defined conditionally on $\mathscr{F}_n$. More precisely, if we let $\| {\bf x} \|_1 = \sum_i |x_i|$ for ${\bf x} \in \mathcal{S}$, then, for any two probability measures $\phi$ and $\chi$ on $\mathbb{R}^\infty$, 
$$d_1(\phi, \chi) = \inf_{{\bf U}, {\bf V}} \mathbb{E}_{n} \left[ \| {\bf U} - {\bf V} \|_1 \right],$$
where ${\bf U}$ is distributed according to $\phi$ and ${\bf V}$ is distributed according to $\chi$, and the infimum is taken over all couplings of $\phi$ and $\chi$. 

Let $\mathcal{R}^{(k)} = \sum_{r=0}^k \sum_{{\bf i} \in \mathcal{A}_r} \Pi_{\bf i} \mathcal{Q}_{\bf i} $ denote the rank of the root node in the delayed weighted branching process constructed using the i.i.d.~vectors $\{ (\mathcal{N}_{\bf i}, \mathcal{Q}_{\bf i}, \{ \mathcal{C}_{({\bf i}, j)}\}_{j \geq 1}): {\bf i} \in \mathcal{U}\}$ (see \cite{chenolvera1} for more details). By Theorem~2  (Case 2) in \cite{chenolvera1}, the convergence of $\hat R_\emptyset^{(n,k)}$ to $\mathcal{R}^{(k)}$ in the Kantorovich-Rubinstein distance will follow once we show that
\begin{equation} \label{eq:KantorovichConv}
d_1(\phi_n^*, \phi^*) + d_1(\phi_n, \phi) \xrightarrow{P} 0  \quad \text{as $n \to \infty$}.
\end{equation}
That convergence in $d_1$ is equivalent to weak convergence plus convergence of the first absolute moments follows from Theorem~6.9 in \cite{Villani_2009}.

Now let $G_n^*(m,q) = \mathbb{P}_{n}(\hat N_\emptyset \leq m, \hat Q_\emptyset \leq q)$, $G_n(m,q,x) = \mathbb{P}_{n}(\hat N \leq m, \hat Q \leq q, \hat C \leq x)$, $G^*(m,q) = P(\mathcal{N} \leq m, \mathcal{Q} \leq q)$, and  $G(m,q,x) = P(\mathcal{N} \leq m, \mathcal{Q} \leq q)P(\mathcal{C} \leq x)$. Note that by Theorem~\ref{T.KRconvergence} we have
$$d_1(G_n, G) + d_1(G_n^*, G^*)  \xrightarrow{P} 0 \qquad n \to \infty.$$
Moreover, $d_1(G_n^*, G^*) = d_1(\phi_n^*, \phi^*)$. To see that $d_1(G_n, G) \xrightarrow{P} 0$ implies that $d_1(\phi_n, \phi) \xrightarrow{P} 0$, choose $(\hat N, \hat Q, \hat C, \mathcal{N}, \mathcal{Q}, \mathcal{C})$ and $(\hat N_\emptyset, \hat Q_\emptyset, \hat C_\emptyset, \mathcal{N}, \mathcal{Q}, \mathcal{C})$ such that $\mathbb{E}_{n}\left[ \| (\hat N, \hat Q, \hat C) - (\mathcal{N}, \mathcal{Q}, \mathcal{C}) \|_1 \right] = d_1(G_n, G)$, which can be done since optimal couplings always exist (see Theorem~4.1 in \cite{Villani_2009}). 
Next, note that since $|\hat C| \leq c $ and $|\mathcal{C}| \leq c$ with $c < 1$, 
\begin{align*}
d_1(\phi_n, \phi) &\leq \mathbb{E}_{n} \left[ \|  (\hat C \hat Q, \hat C 1(\hat N \geq 1), \hat C 1(\hat N \geq 2), \dots) - (\mathcal{C} \mathcal{Q}, \mathcal{C} 1(\mathcal{N} \geq 1), \mathcal{C} 1(\mathcal{N} \geq 2), \dots) \|_1  \right]  \\
&= \mathbb{E}_{n} \left[ |\hat C \hat Q -  \mathcal{C} \mathcal{Q}| + \sum_{i=1}^\infty | \hat C 1(\hat N \geq i) - \mathcal{C} 1(\mathcal{N} \geq i) |  \right] \\
&\leq \mathbb{E}_{n} \left[ |\hat C| |\hat Q - \mathcal{Q}|  + |\mathcal{Q} ||\hat C - \mathcal{C}| + \sum_{i=1}^\infty |\hat C| |  1(\hat N \geq i) -  1(\mathcal{N} \geq i)| + | \hat C  -  \mathcal{C}  | 1(\mathcal{N} \geq i) \right] \\
&\leq  c\mathbb{E}_{n} \left[ |\hat Q - \mathcal{Q}| \right]  + \mathbb{E}_{n}[ |\mathcal{Q} ||\hat C - \mathcal{C}|] + c \sum_{i=1}^\infty \mathbb{E}_{n} \left[ |  1(\hat N \geq i) -  1(\mathcal{N} \geq i)| \right] \\
&\hspace{5mm} + \mathbb{E}_{n} \left[ | \hat C  -  \mathcal{C}  |  \sum_{i=1}^\infty  1(\mathcal{N} \geq i) \right] \\
&= c\mathbb{E}_{n} \left[ |\hat Q - \mathcal{Q}| \right]  + c \sum_{i=1}^\infty \mathbb{E}_{n} \left[ |  1(\hat N < i \leq \mathcal{N}) - 1(\mathcal{N} < i \leq \hat N) | \right] \\
&\hspace{5mm} + \mathbb{E}_{n}  \left[ | \hat C  -  \mathcal{C}  |  (\mathcal{N} +|\mathcal{Q}|) \right] \\
&\leq c \mathbb{E}_{n} \left[ |\hat Q - \mathcal{Q}| \right]  + c \mathbb{E}_{n} \left[ |  \hat N -\mathcal{N} | \right] + \mathbb{E}_{n}  \left[ | \hat C  -  \mathcal{C}  | ( \mathcal{N} + |\mathcal{Q}|) \right] \\
&\leq c d_1(G_n, G) + \mathbb{E}_{n}  \left[ | \hat C  -  \mathcal{C}  | ( \mathcal{N} + |\mathcal{Q}|) \right] . 
\end{align*}
Since $\mathbb{E}_{n}[ |\hat C - \mathcal{C}|] \leq d_1(G_n, G) \xrightarrow{P} 0$, then dominated convergence gives that $\mathbb{E}_{n}  \left[ | \hat C  -  \mathcal{C}  |  (\mathcal{N} +|\mathcal{Q}|) \right] \xrightarrow{P} 0$ as well. 

Finally, it is well known that provided $E\left[ \sum_{i=1}^{\mathcal{N}} |\mathcal{C}_i| \right] = E[\mathcal{N}] E[\mathcal{C}_1] < 1$, we have $\mathcal{R}^{(k)} \to \mathcal{R} = \sum_{r=0}^\infty \sum_{{\bf i} \in \mathcal{A}_r} \Pi_{\bf i} \mathcal{Q}_{\bf i} $ a.s.~as $k \to \infty$ (see, e.g., Lemma~4.1 in \cite{Jel_Olv_12}). To see that the required condition is satisfied note that
\begin{align*}
E[ \mathcal{N} ] E[ |\mathcal{C}_1|] &= E\left[ \frac{W^+}{E[W^+]} \cdot \frac{E[W^+] W^-}{\theta} \right] E\left[ \frac{W^+}{E[W^+]}    \int_{-\infty}^\infty |t| 1(\zeta/(Z^++1) \in dt) \right] \\
&= \frac{E[ W^+ W^-]}{\theta} \cdot E\left[ \frac{W^+}{E[W^+]} \frac{|\zeta|}{Z^+ + 1} \right] \\
&= \frac{E[W^-]}{\theta} \cdot E\left[ \frac{ W^+|\zeta| }{E[W^-] W^+/\theta} \left( 1- e^{-E[W^-] W^+/\theta} \right) \right] \\
&=  E\left[ |\zeta| \left(1 - e^{-E[W^-]W^+/\theta} \right) \right] \leq c < 1,
\end{align*}
where we used the observation that if $Z$ is Poisson with mean $\lambda$ then $E[1/(Z+1)] = \lambda^{-1} ( 1 - e^{-\lambda})$. This completes the proof. 
\end{proof}

\subsection{Proof of Theorem~\ref{T.MainPageRank}}

Finally, we combine the inequality in \eqref{eq:Local}, Theorem~\ref{T.CouplingFinal}, and Theorem~\ref{T.SFPEConvergence} to prove Theorem~\ref{T.MainPageRank}.

\begin{proof}
Fix $\epsilon > 0$ and $k \in \mathbb{N}_+$. Next, construct $(R_\xi^{(n,\infty)}, R_\xi^{(n,k)}, \hat R_\emptyset^{(n,k)})$ using the graph exploration and coupling described in Section~\ref{SS.Coupling}. Then, by \eqref{eq:Local} and Theorem~\ref{T.CouplingFinal}, the following limit holds in probability:
\begin{align*}
&\lim_{n \to \infty} \mathbb{P}_n\left( |R_\xi^{(n,\infty)} - \hat R_\emptyset^{(n,k)}| > \epsilon \right) \\
&\leq \lim_{n \to \infty} \mathbb{P}_n\left( |R_\xi^{(n,\infty)} -  R_\xi^{(n,k)}| > \epsilon/2 \right)  +  \lim_{n \to \infty} \mathbb{P}_n\left( |\hat R_\xi^{(n,k)} - \hat R_\emptyset^{(n,k)}| > \epsilon/2 \right)  \\
&= \frac{2 c^{k}}{\epsilon(1-c)} E[ |Q|] .
\end{align*}
Also, by Theorem~\ref{T.SFPEConvergence} there exists $\mathcal{R}^{(k)}$ such that 
$$\lim_{n \to \infty} \mathbb{P}_n\left( |\hat R_\emptyset^{(n,k)} - \mathcal{R}^{(k)}| > \epsilon \right)  = 0$$
in probability, with $\mathcal{R}^{(k)} \to \mathcal{R}$ a.s.~as $k \to \infty$. Choosing $k$ sufficiently large yields
$$R_\xi^{(n,\infty)} \Rightarrow \mathcal{R}$$
as $n \to \infty$. 
\end{proof}


 \appendix

 \section{Verification of Assumption~\ref{A.Types} for i.i.d.~sequences}
\label{Appx.IIDsequences}

Let $\{ {\bf W}_i \}$ be an i.i.d.~sequence of nonnegative vectors having the same distribution as ${\bf W} = (W^-, W^+)$. Suppose throughout the section that $E[W^- + W^+ + W^- W^+ ] < \infty$. We will now verify that such a sequence always satisfies Assumption~\ref{A.Types} for each of the last three random digraph models in Example~\ref{E.IRG}.

Note that Assumption~\ref{A.Types} (a) and (b) are always satisfied, with almost sure convergence, by the strong law of large numbers, and they only require $E[ W^- + W^+] < \infty$.

Therefore, it suffices to verify Assumption~\ref{A.Types} (c) and (d). 

\bigskip

$\bullet$ {\bf Directed Chung-Lu model:}

For part (c) we have
\begin{align*}
\mathcal{E}_n^{CL} &= \frac{1}{n} \sum_{i=1}^n \sum_{1 \leq j \leq n, j \neq i} \left| \frac{W_i^+ W_j^-}{L_n} \wedge 1 - \frac{W_i^+ W_j^-}{\theta n} \wedge 1 \right|  \leq  \frac{1}{n} \sum_{i=1}^n \sum_{j=1}^n \left| \frac{W_i^+ W_j^-}{L_n} - \frac{W_i^+ W_j^-}{\theta n} \right| \\
&= \frac{1}{\theta}  \cdot \frac{\theta n}{L_n} \cdot  \left| \frac{L_n}{\theta n} - 1 \right| \cdot \frac{1}{n} \sum_{i=1}^n W_i^+ \cdot \frac{1}{n} \sum_{j=1}^n W_j^-.
\end{align*}
Since the strong law of large numbers gives $L_n/(\theta n) \to 1$ a.s.~as $n \to \infty$, we conclude that $\mathcal{E}_n^{CL} \to 0$ a.s.~as $n \to \infty$. 

For part (d) note that we have
\begin{align} \label{eq:MixedMomentUB}
\frac{1}{n}  \sum_{i=1}^n \sum_{1 \leq j \leq n, j \neq i} \sum_{1 \leq k \leq n, k \neq i} p_{ji}^{(n)} p_{ik}^{(n)} &\leq \frac{(\theta n)^2}{L_n^2} \cdot \frac{1}{n} \sum_{i=1}^n \sum_{j=1}^n \sum_{k=1}^n  r_{ji}^{(n)} r_{ik}^{(n)} \\
&\rightarrow \frac{E[W^-] E[W^- W^+] E[W^+]}{\theta^2} \quad \text{a.s.} \notag
\end{align}
as $n \to \infty$. To obtain a lower bound let $(W^{(-,n)}, W^{(+,n)})$, $\hat W^{(-,n)}$, and $\tilde W^{(+,n)}$ be conditionally independent (given $\mathscr{F}_n$) random vectors/variables having distributions $F_n(x,y) = n^{-1} \sum_{i=1}^n 1(W_i^- \leq x, W_i^+ \leq y)$, $F_n^-(x) = n^{-1} \sum_{i=1}^n 1(W_i^- \leq x)$, and $F_n^+(x) = n^{-1} \sum_{i=1}^n 1(W_i^+ \leq x)$, respectively. Now note that we can write
\begin{align*}
&\frac{1}{n}  \sum_{i=1}^n \sum_{1 \leq j \leq n, j \neq i} \sum_{1 \leq k \leq n, k \neq i} p_{ji}^{(n)} p_{ik}^{(n)}  \\
&\geq \frac{1}{n} \sum_{i=1}^n \sum_{j=1}^n \sum_{k=1}^n p_{ji}^{(n)} p_{ik}^{(n)} - \frac{1}{n}  \sum_{i=1}^n p_{ii}^{(n)} \left(  \sum_{j=1}^n  p_{ji}^{(n)}  +   \sum_{1 \leq k \leq n, k \neq i}  p_{ik}^{(n)} \right) \\
&\geq \mathbb{E}_n \left[ \left( \frac{\tilde W^{(+,n)} W^{(-,n)}}{L_n/n} \wedge n  \right) \left( \frac{W^{(+,n)} \hat W^{(-,n)}}{L_n/n} \wedge n  \right) \right] \\
&\hspace{5mm} -  \frac{n}{L_n} \cdot \mathbb{E}_n \left[ \left( \frac{W^{(+,n)} W^{(-,n)}}{L_n} \wedge 1 \right) \left( \hat W^{(+,n)} W^{(-,n)} + W^{(+,n)} \tilde W^{(-,n)} \right)       \right].
\end{align*}
Now note that the observation that Assumption~\ref{A.Types} (a) and (b) are satisfied with almost sure convergence, implies that $\Delta_n = d_1(F_n,F) \to 0$ a.s.~as $n \to \infty$, where $F(x,y) = P(W^- \leq x, W^+ \leq y)$. This implies that 
$$ \mathbb{E}_n \left[ \left( \frac{\tilde W^{(+,n)} W^{(-,n)}}{L_n/n} \wedge n  \right) \left( \frac{W^{(+,n)} \hat W^{(-,n)}}{L_n/n} \wedge n  \right) \right] \rightarrow \frac{E[ W^+] E[W^+W^-] E[W^-]}{\theta^2} \quad \text{a.s.}$$
as $n \to \infty$. To see that the second expectation converges to zero note that
\begin{align*}
R_n &\triangleq \mathbb{E}_n \left[ \left( \frac{W^{(+,n)} W^{(-,n)}}{L_n} \wedge 1 \right) \left( \hat W^{(+,n)} W^{(-,n)} + W^{(+,n)} \tilde W^{(-,n)} \right)       \right] \\
&\leq \mathbb{E}_n \left[ 1(W^{(+,n)} W^{(-,n)} > \sqrt{n})  \left( \hat W^{(+,n)} W^{(-,n)} + W^{(+,n)} \tilde W^{(-,n)} \right)      \right]  \\
&\hspace{5mm} + \frac{\sqrt{n}}{L_n} \mathbb{E}_n \left[ \hat W^{(+,n)} W^{(-,n)} + W^{(+,n)} \tilde W^{(-,n)}       \right]  \\
&= \mathbb{E}_n[ W^{(+,n)}] \mathbb{E}_n\left[ 1(W^{(+,n)} W^{(-,n)} > \sqrt{n}) W^{(-,n)} \right] \\
&\hspace{5mm} +  \mathbb{E}_n[ W^{(-,n)}] \mathbb{E}_n\left[ 1(W^{(+,n)} W^{(-,n)} > \sqrt{n}) W^{(+,n)} \right]  + \frac{2\sqrt{n}}{L_n}  \mathbb{E}_n[ W^{(+,n)}]   \mathbb{E}_n[ W^{(-,n)}] \\
&\rightarrow 0 \quad \text{a.s.},
\end{align*}
as  $n \to \infty$. This completes the lower bound for part (d).

\bigskip

$\bullet$ {\bf Directed generalized random graph:}

For part (c) note that we have
\begin{align*}
\mathcal{E}_n^{GRG} &= \frac{1}{n} \sum_{i=1}^n \sum_{1 \leq j \leq n, j \neq i}  \left| \frac{W_i^+ W_j^-}{L_n + W_i^+ W_j^-}  - \frac{W_i^+ W_j^-}{\theta n} \wedge 1 \right| \\
&\leq \frac{1}{n} \sum_{i=1}^n \sum_{j=1}^n   \left(    \frac{W_i^+ W_j^-}{L_n} \wedge 1 - \frac{W_i^+ W_j^-}{L_n + W_i^+ W_j^-} \right)  + \mathcal{E}_n^{CL} \\
&\leq \frac{n}{L_n} \cdot \frac{1}{n} \sum_{i=1}^n W_i^+ \cdot \frac{1}{n} \sum_{j=1}^n W_j^- - \frac{1}{n} \sum_{i=1}^n \sum_{j=1}^n \frac{W_i^+ W_j^-}{L_n+ W_i^+ W_j^-}  + \mathcal{E}_n^{CL} \\
&= \frac{n}{L_n} \mathbb{E}_n[ W^{(+,n)}] \mathbb{E}_n[ \hat W^{(-,n)}] - \mathbb{E}_n\left[ \frac{W^{(+,n)} \hat W^{(-,n)}}{L_n/n + W^{(-,n)} \hat W^{(+,n)}/n} \right] + \mathcal{E}_n^{CL},
\end{align*}
where $W^{(+,n)}$ and $\hat W^{(-,n)}$ are conditionally independent  (given $\mathscr{F}_n$) random variables having distributions $F^+_n$ and $F^-_n$, respectively. Since both expectations converge a.s.~to $E[W^+] E[W^-]/\theta$ and  $\mathcal{E}_n^{CL} \to 0$ a.s., then $\mathcal{E}_n^{GRG} \to 0$ a.s.~as $n \to \infty$. 

For part (d) we again have that \eqref{eq:MixedMomentUB} holds, so we only need a lower bound. Following similar steps as the ones used for the Chung-Lu model, we obtain
\begin{align*}
&\frac{1}{n}  \sum_{i=1}^n \sum_{1 \leq j \leq n, j \neq i} \sum_{1 \leq k \leq n, k \neq i} p_{ji}^{(n)} p_{ik}^{(n)}  \\
&\geq \mathbb{E}_n \left[  \frac{\tilde W^{(+,n)} W^{(-,n)}}{L_n/n +\tilde W^{(+,n)} W^{(-,n)}/n } \cdot \frac{W^{(+,n)} \hat W^{(-,n)}}{L_n/n + W^{(+,n)} \hat W^{(-,n)}/n} \right] -  \frac{n}{L_n} \cdot R_n.
\end{align*}
The convergence of $d_1(F_n, F)$ now gives
$$ \mathbb{E}_n \left[  \frac{\tilde W^{(+,n)} W^{(-,n)}}{L_n/n +\tilde W^{(+,n)} W^{(-,n)}/n } \cdot \frac{W^{(+,n)} \hat W^{(-,n)}}{L_n/n + W^{(+,n)} \hat W^{(-,n)}/n} \right] \rightarrow \frac{E[ W^+] E[W^+W^-] E[W^-]}{\theta^2} \quad \text{a.s.}$$
as $n \to \infty$ and we already showed that $R_n \to 0$ a.s.~as $n \to \infty$.  This completes the lower bound for part (d).

\bigskip

$\bullet$ {\bf Directed Norros-Reittu model:}

For part (c) we can use the inequality $1 - e^{-x} \leq x \wedge 1$ for $x > 0$ to obtain that
\begin{align*}
\mathcal{E}_n^{NR} &= \frac{1}{n} \sum_{i=1}^n \sum_{1 \leq j \leq n, j \neq i}  \left| 1- e^{-W_i^+ W_j^-/ L_n} - \frac{W_i^+ W_j^-}{\theta n} \wedge 1 \right| \\
&\leq \frac{1}{n} \sum_{i=1}^n \sum_{j=1}^n \left( \frac{W_i^+ W_j^-}{L_n} \wedge 1 - 1 + e^{-W_i^+ W_j^-/ L_n}  \right) + \mathcal{E}_n^{CL} \\
&\leq \frac{n}{L_n} \mathbb{E}_n[ W^{(+,n)}] \mathbb{E}_n[ \hat W^{(-,n)}] - \frac{n}{L_n} \mathbb{E}_n\left[ L_n\left( 1 - e^{-W^{(+,n)} \hat W^{(-,n)}/L_n} \right) \right] + \mathcal{E}_n^{CL}.
\end{align*}
Now use that $\lim_{x \to 0} (1- e^{-x})/x = 1$ and a conditional version of Fatou's lemma to obtain that $\mathbb{E}_n \left[ L_n\left( 1 - e^{-W^{(+,n)} \hat W^{(-,n)}/L_n} \right) \right]  \to E[W^-] E[W^+]/\theta$ a.s.~to conclude that $\mathcal{E}_n^P \to 0$ a.s.~as $n \to \infty$. 

For part (d) we again have that \eqref{eq:MixedMomentUB} holds, so we only need a lower bound. Using the inequality $1- e^{-x} \leq x \wedge 1$ for $x > 0$, we obtain 
\begin{align*}
&\frac{1}{n}  \sum_{i=1}^n \sum_{1 \leq j \leq n, j \neq i} \sum_{1 \leq k \leq n, k \neq i} p_{ji}^{(n)} p_{ik}^{(n)}  \\
&\geq n^2 \mathbb{E}_n \left[  \left( 1 - e^{- \tilde W^{(+,n)} W^{(-,n)}/ L_n} \right) \left(1 - e^{-W^{(+,n)} \hat W^{(-,n)}/ L_n} \right)  \right] -  \frac{n}{L_n} \cdot R_n.
\end{align*}
The convergence of $d_1(F_n, F)$ now gives
\begin{align*}
&n^2 \mathbb{E}_n \left[  \left( 1 - e^{- \tilde W^{(+,n)} W^{(-,n)}/ L_n} \right) \left(1 - e^{-W^{(+,n)} \hat W^{(-,n)}/ L_n} \right)  \right]  \\
&= \frac{n^2}{L_n^2} \cdot \mathbb{E}_n \left[  \tilde W^{(+,n)} W^{(-,n)} W^{(+,n)} \hat W^{(-,n)} \left( \frac{1 - e^{- \tilde W^{(+,n)} W^{(-,n)}/ L_n}}{ \tilde W^{(+,n)} W^{(-,n)}/ L_n}  \right) \left( \frac{1 - e^{-W^{(+,n)} \hat W^{(-,n)}/ L_n}}{W^{(+,n)} \hat W^{(-,n)}/ L_n} \right)  \right] \\
&\rightarrow \frac{E[ W^+] E[W^+W^-] E[W^-]}{\theta^2} \quad \text{a.s.}
\end{align*}
as $n \to \infty$ and we already showed that $R_n \to 0$ a.s.~as $n \to \infty$.  This completes the lower bound for part (d).

\section{Numerical examples} \label{Appx.Numerical}

The second set of results in this appendix consists of numerical experiments illustrating the convergence of the distribution of PageRank to that of the attracting endogenous solution to the SFPE \eqref{eq:SFPE}. In addition, we also compare the empirical tail distribution of PageRank on each of the three models from Example~\ref{E.IRG} to the asymptotic tail distribution of $\mathcal{R}$. Since the tail distribution of $\mathcal{R}$ is proportional to that of the in-degree of the graph ($\mathcal{N} \stackrel{\mathcal{D}}{=} Z^+$), this corresponds to testing the power-law hypothesis. 

To simplify the calculations, we use the original formulation of PageRank, i.e., we set $Q_i = 1-c$ and $\zeta_i = c$ for all $1 \leq i \leq n$, where $n$ is the number of vertices in each graph. Since our main interest is in scale-free graphs, we model $W^+$ and $W^-$ as independent Pareto random variables with parameters $(\alpha, \sigma_\alpha)$ and $(\beta, \sigma_\beta)$, respectively. To satisfy Assumption~\ref{A.Types}, we choose the shape parameters $\alpha > 1$ and $\beta \geq 2$.   In all our experiments, we use the same distribution for the generic type vector $(W^+, W^-)$ in all three models, in which case, the limiting $\mathcal{R}$ is the same for all cases. 

In the first set of results, shown in Figure~\ref{fig:cdfs}, we plot the empirical cumulative distribution functions (CDFs) of the PageRank of a randomly chosen node in each of the three models and compare it to the distribution of $\mathcal{R}$.
More precisely, for each model we generate $15000$ independent graphs, each having $n = 5000$ vertices, and compute their ranks using matrix iterations, i.e., we compute
$${\bf R}^{(n,k)} = \sum_{i=0}^k {\bf Q} {\bf M}^i,$$
where ${\bf Q}$ and ${\bf M}$ are defined in Section~\ref{SSS.LocalNeighborhood}. We choose $k$ large enough to ensure that $\| {\bf R}^{(n,\infty)} - {\bf R}^{(n,k)} \|_1 < \epsilon$, with $\epsilon = 0.01$. We then collect the value of $R_1^{(n,k)}$ from each of the 15000 graphs and use these observations to construct the empirical distribution function of $R_1^{(n,k)} \approx R_1^{(n,\infty)}$. Note that since the types are i.i.d., then all the vertices in the graph have the same distribution, and therefore $R_1^{(n,\infty)} \stackrel{\mathcal{D}}{=} R_\xi^{(n,\infty)}$, where $\xi$ is a vertex in the graph uniformly chosen at random. 

To estimate the distribution of $\mathcal{R}$, we use the Population Dynamics algorithm described in \cite{Chen_Olvera_2013} using the generic branching vector $(1-c, \mathcal{N}, \{ \mathcal{C}_i \})$, where the $\{ \mathcal{C}_i\}$ are i.i.d., are independent of $\mathcal{N}$, and have distribution
$$P(\mathcal{C} \in dt) = E\left[ \frac{W^-}{E[W^-]} 1(c/(1+ Z^-) \in dt) \right],$$
with $Z^-$ a mixed Poisson  random variable with mixing parameter $W^- E[W^+]/\theta$,  and $\mathcal{N}$ a mixed Poisson with mixing parameter $W^+ E[W^-]/\theta$. To generate samples of $\mathcal{C}$ we first note that if $\mathcal{Y}$ is a mixed Poisson random variable with mixing parameter $E[W^+] U/\theta$, with $U$ a Pareto random variable with parameters $(\beta-1, \sigma_\beta)$, then $P(\mathcal{Y} = y) = E[W^- 1(Z^- = k)]/E[W^-]$. Hence, $\mathcal{C} \stackrel{\mathcal{D}}{=} c/(\mathcal{Y}+1)$, and we can generate $\mathcal{C}$ by sampling $\mathcal{Y}$.  In addition, the Population Dynamics algorithm requires two parameters, the depth of the recursion $k$ and the size of the pool $m$, which were set to be $k = 9$ and $m = 15000$. 

The values of $\alpha, \sigma_\alpha, \beta, \sigma_\beta$ and the damping factor $c$, as well as the mean degree $\mu = E[\mathcal{N}] = E[W^+] E[W^-]/\theta$, are indicated in each plot.  As we can see in Figure~\ref{fig:cdfs}, the fit of $\mathcal{R}$ to $R_\xi^{(n,\infty)}$ is very good for all three models.

\captionsetup[figure]{labelfont=bf}
\begin{figure}[th]
\centering
\subfloat[$\alpha=1.5, \beta=2.5, \sigma_{\alpha}=2, \sigma_{\beta}=5, \mu=3.49, c=0.85$]{
  \includegraphics[width=0.49\textwidth]{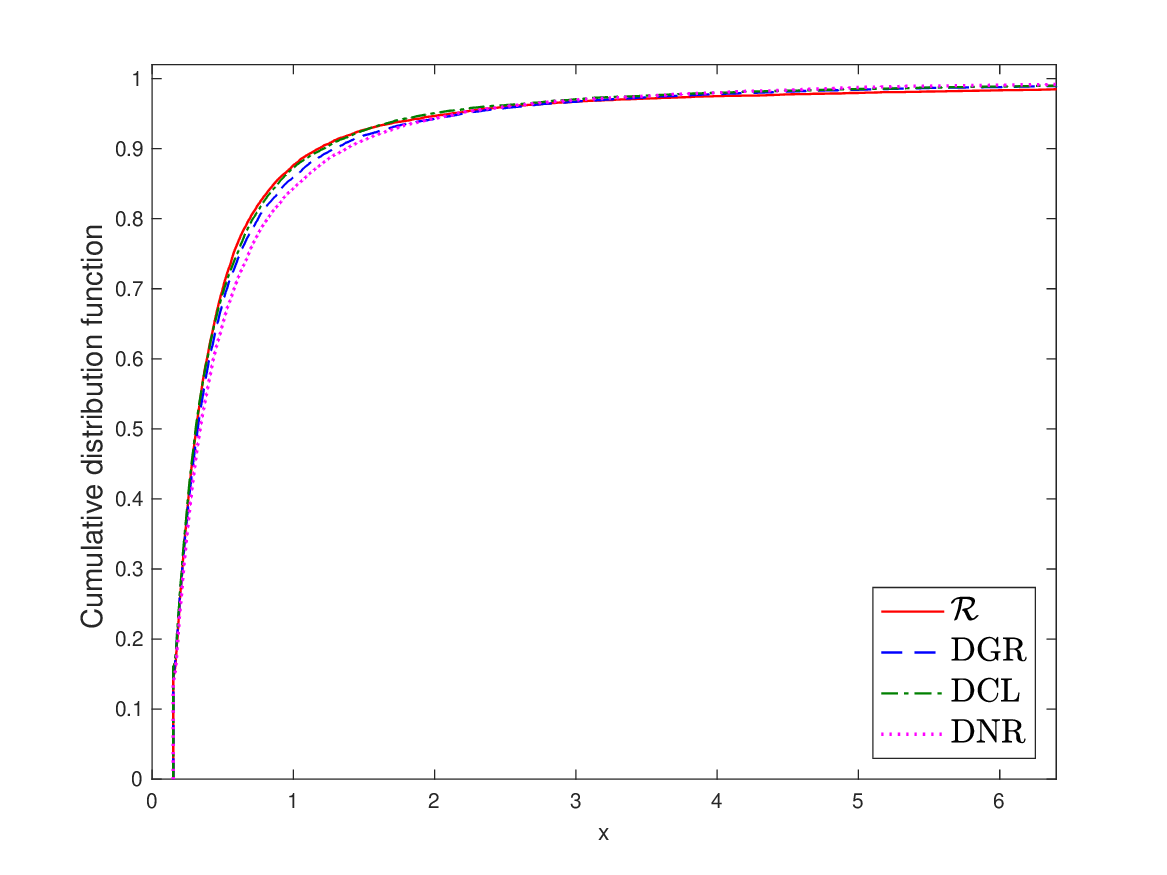}
  \label{subfig:cdf1}
  }
\subfloat[$\alpha=1.8, \beta=2.8, \sigma_{\alpha}=\frac{40}{9}, \sigma_{\beta}=\frac{45}{7}, \mu=5, c=0.45$]{
  \includegraphics[width=0.49\textwidth]{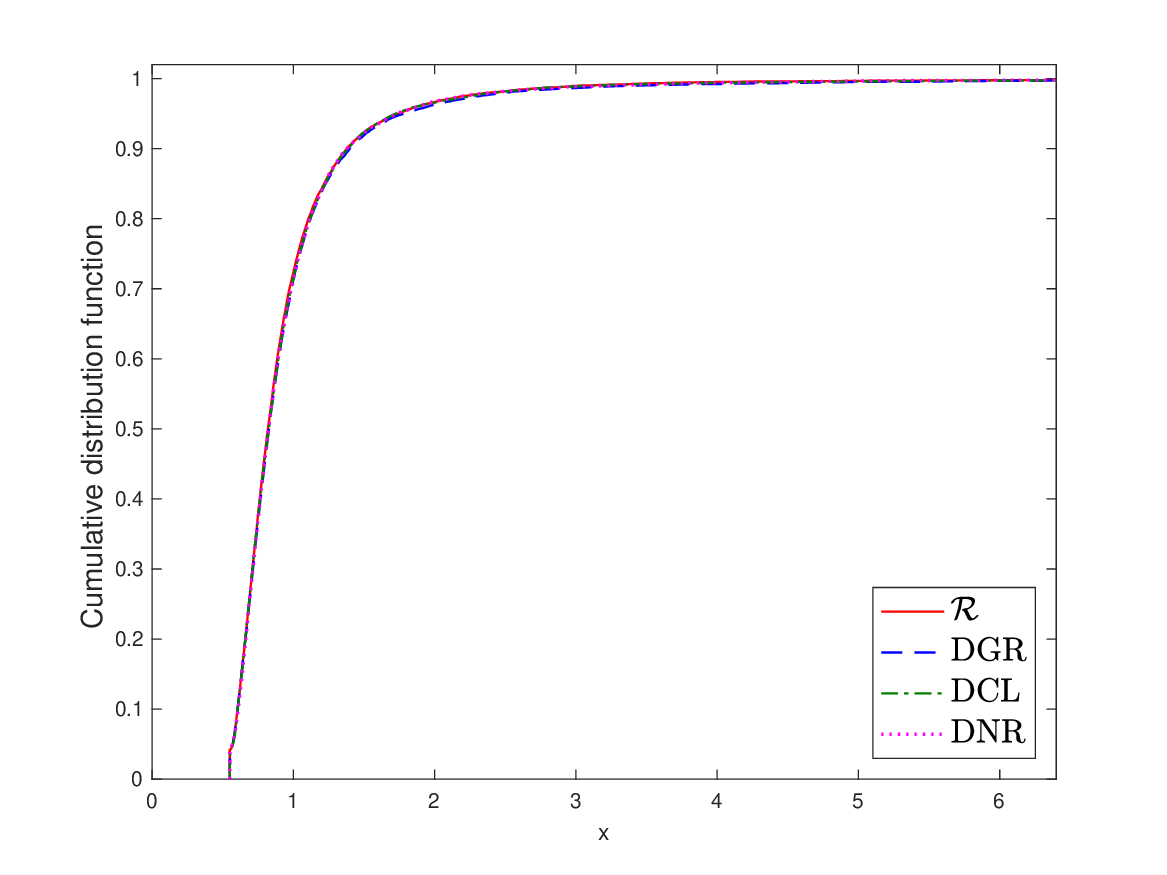}
  \label{subfig:cdf2}
  }
\caption{The CDF of $\mathcal{R}$ compared to the empirical CDFs of $R_\xi^{(n,\infty)}$, $n = 5000$, in each of the three models: directed generalized random graph, directed Chung-Lu model, directed Norros-Reittu model.}
\label{fig:cdfs}
\end{figure}

The second set of experiments compares the empirical tail CDFs of $R_\xi^{(n,\infty)}$ in each of the three models to the asymptotic tail distribution of $\mathcal{R}$, which is given by
$$P(\mathcal{R} > x) \sim  \frac{((1-c) E[\mathcal{C}])^\alpha}{(1-\rho)^\alpha (1-\rho_\alpha)} \cdot P(\mathcal{N} > x), \qquad x \to \infty,$$
where $\rho = E[\mathcal{N}] E[ \mathcal{C}] = c E[ 1- e^{-E[W^+] W^+/\theta}]$ and $\rho_\alpha = E[\mathcal{N}] E[\mathcal{C}^\alpha]$ (see Theorem~5.1 in \cite{Jel_Olv_10}).  The empirical CDFs of $R_\xi^{(n,\infty)}$ were computed as before, using 15000 independent graphs for each model.  Figure~\ref{fig:tails} shows the corresponding log-log plots. As we can see, the power-law hypothesis holds reasonably well, with the poorer fit towards the end of the tail, which is to be expected considering the finite nature of the sample.

 \begin{figure}[th]
\centering
\subfloat[$\alpha=1.5, \beta=2.5, \sigma_{\alpha}= 2, \sigma_{\beta}=5, \mu=3.49, c=0.85$]{
  \includegraphics[width=0.49\textwidth]{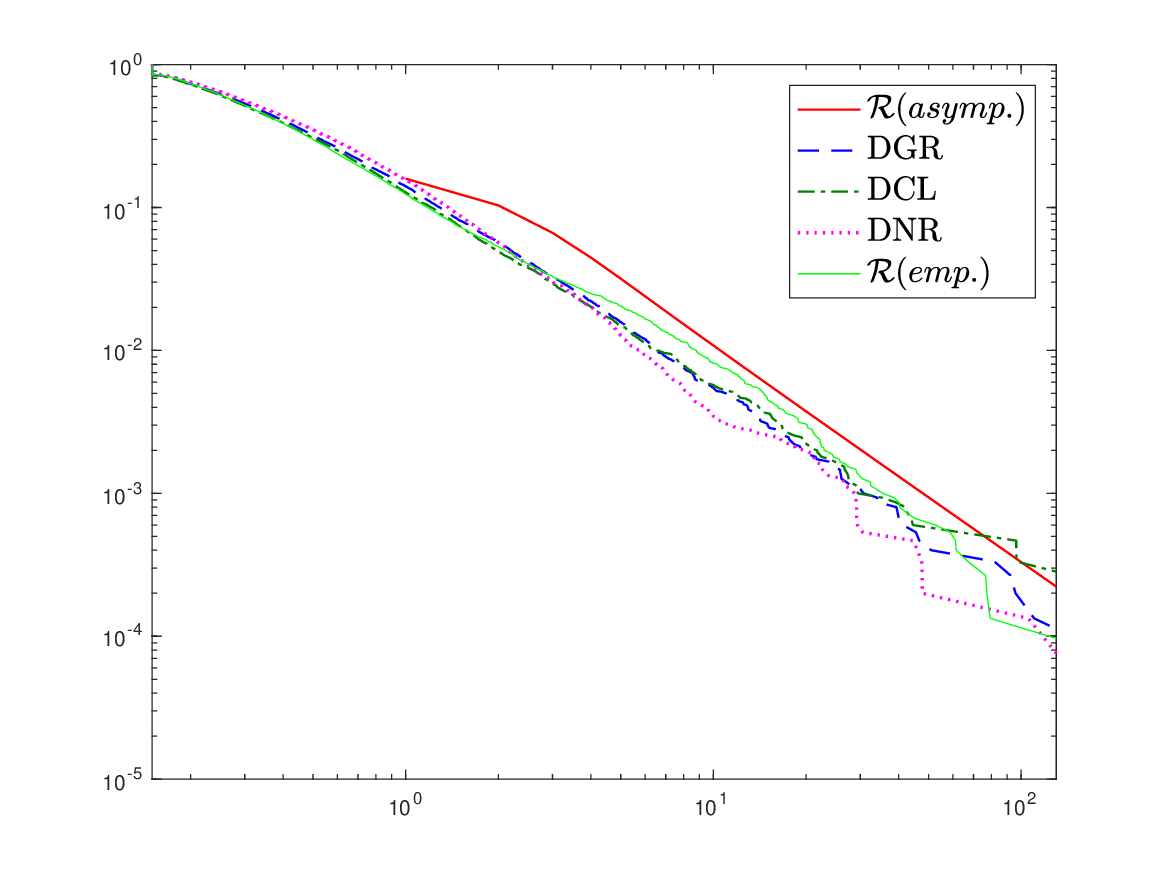}
  \label{subfig:tail1}
  }
\subfloat[$\alpha=1.8, \beta=2.8, \sigma_{\alpha}=\frac{40}{9} , \sigma_{\beta}=\frac{45}{7},  \mu=5, c=0.45$]{
  \includegraphics[width=0.49\textwidth]{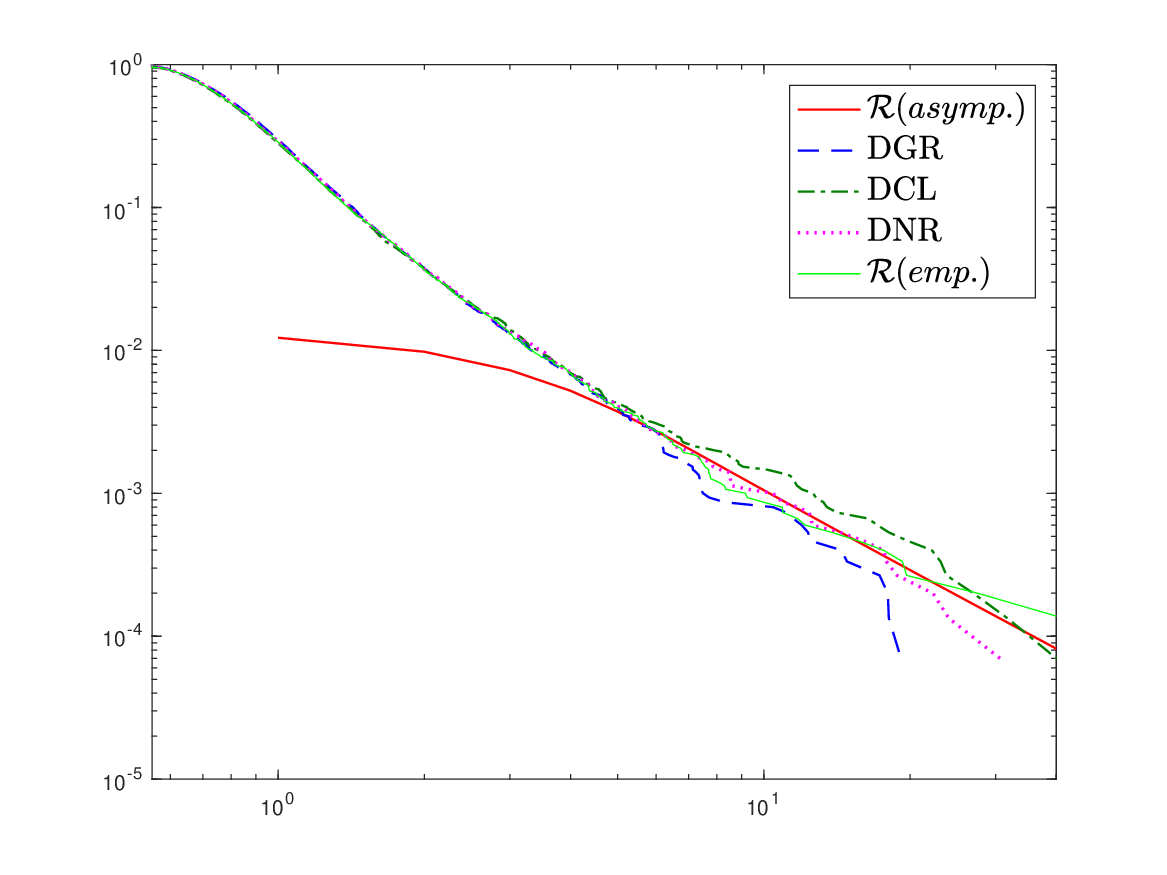}
  \label{subfig:tail2}
  }
\caption{Log-log plot of the asymptotic tail distribution of $\mathcal{R}$ compared to the empirical tail CDFs of $R_\xi^{(n,\infty)}$, $n = 5000$, in each of the three models: directed generalized random graph, directed Chung-Lu model, directed Norros-Reittu model. The plot for $\mathcal{R}(emp.)$ was generated via the Population Dynamics algorithm as for Figure~\ref{fig:cdfs}.}
\label{fig:tails}
\end{figure}

\bibliographystyle{elsarticle-num} 
\bibliography{references}


\end{document}